\documentclass[11pt]{amsart}
\usepackage{amsmath,amsthm,amsfonts,amssymb}
\usepackage[alphabetic]{amsrefs}
\usepackage[all]{xy}
\usepackage{amssymb,amsthm}
\usepackage{mathrsfs}
\usepackage{tikz}
\usepackage{color}
\usepackage{graphicx,pict2e}
\usepackage[
	hyperindex,
	pagebackref,
	pdftex,
	pdfdisplaydoctitle,
	pdfpagemode=UseNone,
	breaklinks=true,
	extension=pdf,
	bookmarks=false,
	plainpages=false,
	colorlinks,
	linkcolor=linkblue,
	citecolor=linkblue,
	urlcolor=linkblue,
	pdfmenubar=true,
	pdftoolbar=true,
	pdfpagelabels,
	pdfpagelayout=SinglePage,
	pdfview=Fit,
	pdfstartview=Fit
]{hyperref}
\definecolor{linkblue}{rgb}{0,0.2,0.6}
\newcommand{\mc}[1]{\mathcal{#1}}
\newcommand{\ov}[1]{\overline{#1}}
\newcommand{\op}[1]{\operatorname{#1}}
\newcommand{\ovop}[1]{\ov{\op{#1}}}
\newcommand{\Vqd}[1]{\mathcal{#1}}
\newcommand{\Udnga}[3]{V^{#1}_{#2,#3}}

\newcommand{\PP}{\mathbb{P}}

\newcommand{\ZZ}{\mathbb{Z}}

\newcommand{\cO}{\mathcal{O} }

\def\Mzn{\ovop{M}_{0,n} }

\newcommand{\ovmc}[1]{\ov{\mc{#1}}}

\newcommand{\Z}{\mathbb{Z}}

\newcommand{\SL}{\operatorname{SL}}
\newcommand{\sL}{\mathfrak{sl}}

\newcommand{\F}{\operatorname{F}}

\newtheorem{theorem}{Theorem}[section]
\newtheorem{corollary}[theorem]{Corollary}

\newtheorem{conjecture}[theorem]{Conjecture}
\newtheorem{lemma}[theorem]{Lemma}

\newtheorem{example}[theorem]{Example}
\newtheorem{proposition}[theorem]{Proposition}
\newtheorem{definition}[theorem]{Definition}
\theoremstyle{definition}
\newtheorem{remark}[theorem]{Remark}

\newcommand{\git}{\ensuremath{\operatorname{/\!/}}}
\newcommand{\quotientname}{Veronese quotient}
\newcommand{\quotientnames}{Veronese quotients}

\begin{document}

\pagenumbering{arabic}
\title[Gibney, Jensen, Moon, Swinarski]{Veronese quotient models of $\ovop{M}_{0,n}$ and conformal blocks}
\author{A. Gibney, D. Jensen, H-B. Moon, D. Swinarski}
\date{\today}

\begin{abstract}The moduli space $\ovop{M}_{0,n}$ of
  Deligne-Mumford stable $n$-pointed rational curves admits morphisms
  to spaces recently constructed by Giansiracusa, Jensen, and Moon that we call Veronese quotients.   We study divisors on $\ovop{M}_{0,n}$ associated to these maps and show that these divisors arise as first Chern classes of vector bundles of conformal blocks.
\end{abstract}

\maketitle

\section*{Introduction} The moduli space of Deligne-Mumford stable n-pointed rational curves $\ovop{M}_{0,n}$ is a natural compactification of
the moduli space of smooth pointed curves, and has figured prominently in the literature.     A central motivating question is to describe other compactifications of $\op{M}_{0,n}$ that receive morphisms from $\ovop{M}_{0,n}$.  From the perspective of Mori theory,  this is tantamount to describing certain semi-ample divisors on $\ovop{M}_{0,n}$.  This work is concerned with two recent constructions that each yield an abundance of such semi-ample divisors on $\ovop{M}_{0,n}$, and the relationship between them.  The first comes from Geometric Invariant Theory (GIT), while the second from conformal field theory.

There are new natural birational models of $\ovop{M}_{0,n}$ obtained via GIT which are moduli spaces of pointed rational normal curves of a fixed degree $d$, where the curves and the marked points are weighted by nonnegative rational numbers $(\gamma, A) = (\gamma,(a_1,\cdots,a_n))$ \cite{Giansiracusa, GiansiracusaSimpson,GJM}.  These so-called \emph{\quotientnames}\ $\Udnga{d}{\gamma}{A}$ are particularly remarkable as they specialize to nearly every known compactification of $\op{M}_{0,n}$ \cite{GJM}.  There are birational morphisms from $\ovop{M}_{0,n}$ to these GIT quotients, and their natural polarization can be pulled back along this morphism, yielding a seemingly endless supply of semi-ample divisors $\Vqd{D}_{\gamma,A}$  on $\ovop{M}_{0,n}$.

A second recent development in the birational geometry of $\ovop{M}_{0,n}$ involves divisors that arise from conformal field theory.  These divisors are first Chern classes of vector bundles of conformal blocks on the moduli stack  $\ovmc{M}_{g,n}$.   Constructed using the representation theory of affine Lie algebras \cite{TUY, Fakh}, these vector bundles depend on the choice of a simple Lie algebra $\mathfrak{g}$, a positive integer $\ell$,  and an $n$-tuple $\vec{\lambda} = (\lambda_{1}, \cdots, \lambda_{n})$ of dominant integral weights in the Weyl alcove for $\mathfrak{g}$ of level $\ell$.     
Vector bundles of conformal blocks are globally generated when $g=0$ \cite{Fakh}*{Lemma 2.5}, and their first Chern classes $c_1(\mathbb{V}(\mathfrak{g},\ell, \vec{\lambda}))=\mathbb{D}(\mathfrak{g},\ell, \vec{\lambda})$, the conformal block divisors, are semi-ample.

When $\gamma = 0$, it was shown in \cites{Giansiracusa, GiansiracusaGibney} that the divisors $\Vqd{D}_{0, A}$ coincide with conformal block divisors for type A and level one.    Our guiding philosophy is that there is a general correspondence between \quotientnames\ and conformal block divisors.    After giving background on \quotientnames\  in \S $1$,  in support of this, we:
\begin{enumerate}
\item  give a new modular interpretation for a particular family of \quotientnames\ (\S \ref{s:NewSpace});
\item  derive intersection numbers for all $\Vqd{D}_{\gamma,A}$ with curves on $\ovop{M}_{0,n}$ (Theorem \ref{thm:intersectionFormula});
\item  show that the models described in \S \ref{s:NewSpace} are given by conformal blocks divisors (Theorem \ref{main});
\item  provide several conjectures (and supporting evidence) generalizing these results (\S \ref{conjectures}).
\end{enumerate}

In order to further motivate and put this work in context, we next say a bit more about $(1)$-$(4)$.

\subsection{A new modular interpretation for a particular family of  \quotientnames\ }  Much work has focused on alternative compactifications of $\op{M}_{0,n}$ \cites{KapChow, KapVer, Boggi, LosevManin, HassettWeighted, Simpson,  Smyth, FedorchukCyclic, GiansiracusaSimpson, Giansiracusa, GJM}.    As was shown in \cite{GJM}, every choice of allowable weight data for \quotientnames\ (Definition \ref{allowable}) yields such a compactification, and nearly every previously known compactification arises as such a Veronese quotient.  It is interesting to see what types of compactifications can be constructed as Veronese quotients, particularly for linearizations on the GIT walls, where it is possible that the quotient may not coincide with previously constructed compactifications of $\op{M}_{0,n}$.   
In \S \ref{s:NewSpace}, we study the particular \quotientnames\  $\Udnga{g+1-\ell}{\frac{\ell-1}{\ell+1}}{(\frac{1}{\ell+1})^{2g+2}}$ for $\ 1\le \ell \le g$. 
In Theorem \ref{NewSpace}, we provide a new modular interpretation for these spaces and we note that, prior to \cite{GJM}, this moduli space had not appeared in the literature (see Remark \ref{TrulyNew}).  Our main application is to show that the nontrivial conformal blocks divisors $\mathbb{D}(\sL_2,\ell,\omega_{1}^{2g+2})$ are pullbacks of ample classes from these Veronese quotients.  For this, we prove several results concerning morphisms between these \quotientnames\ (see Corollary \ref{phiell} and Proposition \ref{prop:morphism}).

\subsection{The classes of all Veronese quotient divisors $\Vqd{D}_{\gamma,A}$}
The Veronese quotients $\Udnga{d}{\gamma}{A}$ receive birational morphisms  $\varphi_{\gamma, \mc{A}}$ from $\ovop{M}_{0,n}$, and in \S \ref{wpvc} we study the divisors $\Vqd{D}_{\gamma,A}=\varphi_{\gamma, \mc{A}}^*(L_{\gamma,\mc{A}})$, where $L_{\gamma,\mc{A}}$ is the natural ample polarization on $\Udnga{d}{\gamma}{A}$.  In Theorem \ref{thm:intersectionFormula}, we give a formula for the intersection of $\Vqd{D}_{\gamma,A}$ with $\op{F}$-curves (Def. \ref{FCurve}), a collection of curves that span the vector space of 1-cycles on $\ovop{M}_{0,n}$.    Theorem \ref{thm:intersectionFormula}, which is a vast generalization of formulas that have appeared for  $d\in \{1,2\}$\cites{AlexeevSwinarski,GiansiracusaSimpson} and for $\gamma=0$ \cites{Giansiracusa, GiansiracusaGibney}, captures a lot of information
about the nef cone $\op{Nef}(\ovop{M}_{0,n})$.   For example, since adjacent chambers in the GIT cone correspond to adjacent faces of the nef cone, by combining Theorem \ref{thm:intersectionFormula} with the results of \cite{GJM}, we could potentially describe many faces of $\op{Nef}(\ovop{M}_{0,n})$.  Moreover, Theorem \ref{thm:intersectionFormula} is equivalent to giving the class of $\Vqd{D}_{\gamma,A}$ in the N\'eron Severi space.  To illustrate this, we explicitly write down the classes of the conformal divisors with $\op{S}_n$-invariant choices of weights (Corollary \ref{class}) and the particularly
simple formula for the divisors that give rise to the maps to the \quotientnames\  $\Udnga{g+1-\ell}{\frac{\ell-1}{\ell+1}}{(\frac{1}{\ell+1})^{2g+2}}$ (Example \ref{JensenApp}).

\subsection{A particular family of conformal blocks divisors}   In \cites{Giansiracusa, GiansiracusaGibney} it was shown that the divisors $\Vqd{D}_{0, A}$ coincide with conformal blocks divisors of type A and level one, and  in \cite{ags} it is shown that the divisors $\mathbb{D}(\sL_2,\ell, \omega_1^{2g+2})$ and $\Vqd{D}_{\frac{\ell-1}{\ell+1},(\frac{1}{\ell+1})^{2g+2}}$ are proportional for the two special cases $\ell = 1$ and $g$.  In \cite{ags} the authors ask whether there is a more general correspondence between $\mathbb{D}(\sL_2,\ell,\omega_1^{n})$ and Veronese quotient divisors. Theorem \ref{main} gives an affirmative answer to their question and a complete answer for the $\sL_2$, $\omega_1^n$ case (cf. Remark \ref{rmk:sl2criticallevel}).  One of the main insights in this work is that, while not proportional for the remaining levels $\ell \in \{2, \ldots, g-1\}$, the divisors $\mathbb{D}(\sL_2,\ell, \omega_1^{2g+2})$ and  $\Vqd{D}_{\frac{\ell-1}{\ell+1},(\frac{1}{\ell+1})^{2g+2}}$  lie on the same face of the nef cone of $\ovop{M}_{0,n}$.   In other words, the two semi-ample divisors define maps to isomorphic birational models of $\ovop{M}_{0,n}$. The corresponding birational models are precisely the spaces described in \S \ref{s:NewSpace}.

\subsection{Generalizations}
Evidence suggests that type A conformal blocks divisors with nonzero weights give rise to compactifications of $\op{M}_{0,n}$, and that these compactifications coincide with Veronese quotients.  This is certainly true for $\ell=1$, and for the family of higher level $\sL_2$ divisors considered in this paper, as well as for a large number of cases found using \cite{ConfBlocks}, software written for Macaulay 2 by David Swinarski.  In \S \ref{5a} we provide evidence in support of these ideas in the $\sL_2$ cases.  In \S \ref{5b} we describe consequences of and evidence for Conjecture \ref{PreserveInterior}, which asserts that conformal blocks divisors (with strictly positive weights) separate points on $\op{M}_{0,n}$.

\subsection*{Acknowledgements} We would like to thank Valery Alexeev, Maksym Fedorchuk, Noah Giansiracusa, Young-Hoon Kiem,  Jason Starr, and Michael Thaddeus for many  helpful and inspiring discussions.  The first author is supported on NSF DMS-1201268.

\section{Background on \quotientnames\ }

We begin by reviewing general facts about \quotientnames, including a description of them as moduli spaces of weighted pointed (generalized) Veronese curves (Section \ref{spaces}), and the morphisms $\varphi_{\gamma, A}: \ovop{M}_{0,n}\longrightarrow \Udnga{d}{\gamma}{A}$ (Section \ref{morphisms}), from \cite{GJM}.

\subsection{The spaces $\Udnga{d}{\gamma}{A}$}\label{spaces}
Following \cite{GJM}, we write $Chow(1,d,\mathbb{P}^d)$ for the irreducible component of the Chow variety parameterizing curves of degree $d$ in $\mathbb{P}^d$ and their limit cycles, and consider the incidence correspondence
$$U_{d,n}:=\{(X,p_1,\cdots,p_n) \in Chow(1,d,\mathbb{P}^d) \times (\mathbb{P}^d)^n : p_i \in X \;\forall i\}.$$
There is a natural action of $\SL(d+1)$ on $U_{d,n}$, and one can form the GIT quotients $U_{d,n}\git_{L} \SL(d+1)$ where $L$ is a $\SL(d+1)$-linearized ample line bundle.  The Chow variety and each copy of $\mathbb{P}^{d}$ has a tautological ample line bundle $\cO_{Chow}(1)$ and $\cO_{\PP^{d}}(1)$, respectively.  By taking external tensor products of them, for each sequence of positive rational numbers $(\gamma, (a_{1}, \cdots, a_{n}))$, we obtain a $\mathbb{Q}$-linearized ample line bundle $L = \cO(\gamma) \otimes \cO(a_{1}) \otimes \cdots \otimes \cO(a_{n})$.


\begin{definition}\label{allowable}We say that a linearization $L$ is {\emph{\textbf{allowable}}} if it is an element of the set $\Delta^0$, where
$$\Delta^0 = \{ (\gamma, A)=(\gamma, (a_1,\cdots,a_n)) \in \mathbb{Q}^{n+1}_{\ge 0}: \gamma<1, \ 0<a_i<1, \mbox{ and } \
(d-1) \gamma + \sum_{i}a_i= d+1 \}.$$
If $a_{1} = \cdots  = a_{n} = a$, then we write $a^{n}$ for $A =
(a_{1}, \cdots, a_{n})$.

We call a GIT quotient of this form a \emph{\quotientname}, denoted
\begin{displaymath} \Udnga{d}{\gamma}{A} := U_{d,n}\git_{\gamma,A} \SL(d+1).
\end{displaymath}
\end{definition}

Given $(X,p_1,\cdots,p_n) \in U_{d,n}$, we may think of a choice $L \in \Delta^0$ as an assignment of a rational weight $\gamma$ to the curve $X$ and another weight $a_i$ to each of the marked points $p_i$.  The conditions $\gamma<1$ and $0<a_i<1$ for all $i$ imply that the quotient $U_{d,n}\git_{\gamma,A} \SL(d+1)$ is a compactification of $\op{M}_{0,n}$ \cite[Proposition 2.10]{GJM}.   As is reflected in Lemma \ref{modular} below, the quotients have a modular interpretation parameterizing pointed degenerations of Veronese curves.

By taking $d=1$, and $\gamma=0$, one obtains the GIT quotients $(\mathbb{P}^{1})^{n}\git_{A}\SL(2)$ with various weight data $A$.  This quotient, which appears in \cite[Chapter 3]{git} under the heading ``an elementary example'', has been studied by many authors.  It was generalized first to $d=2$ by Simpson in \cite{Simpson} and later Giansiracusa and Simpson in \cite{GiansiracusaSimpson}, and then for arbitrary $d$, and $\gamma=0$ by Giansiracusa  in \cite{Giansiracusa}.  More generally, the quotients for arbitrary $d$ and $\gamma \geq 0$ are defined and studied by Giansiracusa, Jensen and Moon in \cite{GJM}.

The semistable points of  $U_{d,n}$ with respect to the linearization $(\gamma, A)$ have the following nice geometric properties.

\begin{lemma}\cite[Corollary 2.4, Proposition 2.5, Corollary 2.6, Corollary 2.7]{GJM}\label{modular}  For an allowable choice $(\gamma, A)$ (Definition. \ref{allowable}), a semistable point $(X,p_1,\cdots,p_n)$ of $U_{d,n}$  has the following properties.
 \begin{enumerate}
 \item  $X$ is an arithmetic genus zero curve having at worst multi-nodal singularities.
 \item Given a subset $J \subset \{1,\cdots, n\}$, the marked points $\{p_j: j\in J\}$
 can coincide at a point of multiplicity $m$ on $X$ as long as
\[
	(m-1)\gamma + \sum_{j \in J}a_j \le 1.
\]
 In particular, a collection of marked points can coincide at a at a smooth point of $X$ as long as their total wieght is at most $1$.  Also, a semistable curve cannot have a singularity of multiplicity $m$ unless $\gamma \le \frac{1}{m-1}$.
 \end{enumerate}
 \end{lemma}

\subsection{The morphisms $\varphi_{\gamma, A}: \ovop{M}_{0,n}\longrightarrow \Udnga{d}{\gamma}{A}$}\label{morphisms}
In \cite{GJM} the authors prove the existence and several properties of birational morphisms from $\Mzn$ to \quotientnames.

\begin{proposition}\label{prop:MorThroughHassett}
\cite[Theorem 1.2, Proposition 4.7]{GJM} For an allowable choice $(\gamma, A)$, there exists a regular birational map $\varphi_{\gamma,A}: \ovop{M}_{0,n} \rightarrow \Udnga{d}{\gamma}{A}$ preserving the interior $\op{M}_{0,n}$.  Moreover, $\varphi_{\gamma, A}$ factors through the contraction maps $\rho_{A}$ to Hassett's moduli spaces $\ovop{M}_{0, A}$:
 \[
	\xymatrix{\Mzn \ar[d]_{\rho_{A}} \ar[rd]^{\varphi_{\gamma,A}}\\
	\ovop{M}_{0, A}
	\ar[r]^{\phi_{\gamma}}&
	\Udnga{d}{\gamma}{A}.}
\]
\end{proposition}

By the definition of the projective GIT quotient, there is a natural choice of an ample line bundle on each GIT quotient.  By pulling it back to $\Mzn$, we obtain a semi-ample  divisor.

 \begin{definition}\label{def:DgammaA}
Let $L = (\gamma, A)$ be an allowable linearization on $U_{d,n}$, and let $\overline{L} = L\git_{L}\SL(d+1)$ be the natural $\mathbb{Q}$-ample line bundle on $\Udnga{d}{\gamma}{A}$.  Define $\mathcal{D}_{\gamma,A}$ to be the semi-ample line bundle $\varphi_{\gamma, A}^{*}(\overline{L})$.
\end{definition}

Next, following  \cite{HassettWeighted} and \cite{GJM}, we describe the $\F$-curves contracted by $\varphi_{\gamma, A}$ and $\rho_{A}$.  To do this, we first define $\F$-curves, which together span the vector space $N_1(\ovop{M}_{0,n})$ of numerical equivalence classes of $1$-cycles on $\ovop{M}_{0,n}$.

\begin{definition}\label{FCurve}Let $A_1 \sqcup A_2 \sqcup A_3 \sqcup A_4 = [n]=\{1,\cdots,n\}$ be a partition, and set $n_{i} = |A_{i}|$.  There is an embedding
$$f_{A_1,A_2,A_3,A_4}: \ovop{M}_{0,4} \longrightarrow \ovop{M}_{0,n}$$
given by attaching four  legs $L(A_{i})$ to $(X, (p_1,\cdots,p_4)) \in \ovop{M}_{0,4}$ at the marked points.  More specifically, to each $p_i$ we attach a stable $n_i+1$-pointed fixed rational curve $L(A_{i})=(X_i,(p^i_1, \cdots ,p^i_{n_i}, p^i_a))$ by identifying $p^i_a$ and $p_i$, while if $n_i=1$ for some $i$, we just keep $p_i$ as is.  The image is a curve in $\ovop{M}_{0,n}$ whose equivalence class is the {\bf{$\F$-curve}} denoted $F(A_{1}, A_{2}, A_{3}, A_{4})$.  Each member of the $\F$-curve consists of a (varying) {\bf{spine}} and 4 (fixed) {\bf{legs}}.

In many parts of this paper, we will focus on symmetric divisors and $\F$-curves, in which case the equivalence class is determined by the number of marked points on each leg.  In this case, we write $F_{n_{1}, n_{2}, n_{3}, n_{4}}$ for \emph{any} $\F$-curve class $F(A_{1}, A_{2}, A_{3}, A_{4})$ with $|A_{i}| = n_{i}$.
\end{definition}

The $\F$-curves $F(A_{1}, A_{2}, A_{3}, A_{4})$ contracted by the Hassett morphism $\rho_{A}$ are precisely those for which one of the legs, say $L(A_{i})$, has weight $\sum_{j \in A_i}a_j \ge \sum_{j \in [n]}a_j -1$.  We can always order the cells of the partition so that $A_4$ is the heaviest-- that is, $\sum_{j\in A_4}a_j \ge \sum_{j \in A_i}a_j$, for all $i$.  In other words, $L(A_{4})$ is the heaviest leg.  As the morphism $\varphi_{\gamma, A}$ factors through $\rho_{A}$, these curves are also contracted by $\varphi_{\gamma, A}$.  This morphism may contract additional $\F$-curves as well, which we describe here.

As is proved in \cite{GJM}, the map $\varphi_{\gamma, A}$ contracts those curves $F(A_{1}, A_{2}, A_{3}, A_{4})$ for which the sum of the degrees of the four legs is equal to $d$. One can use a function $\sigma$, given in Definition \ref{sigma}, to compute the degree of the legs of an $\F$-curve.

\begin{definition}\cite[Section 3.1]{GJM}\label{sigma} Consider the function $\phi: 2^{[n]} \times \Delta^0 \longrightarrow \mathbb{Q}$,  given by
\[
\phi(J,\gamma,A) = \frac{a_J-1}{1-\gamma}, \mbox{ where for } \  J \in 2^{[n]},  \ a_J=\sum_{j\in J}a_j.
\]
 For a fixed allowable linearization $(\gamma, A) = (\gamma, (a_1,\cdots, a_n))$ (cf. Definition \ref{allowable}), let
\[
\sigma(J)=\left\{\begin{matrix}
\lceil \phi(J,\gamma, A) \rceil & \text{if } 1 < a_{J} < a_{[n]}-1,\\
0 & \text{if } a_{J} < 1,\\
d & \text{if } a_{J} > a_{[n]}-1.\\
\end{matrix}\right.
\]
Finally, for $(X, p_{1}, \cdots, p_{n}) \in U_{d,n}$ and $E \subset X$ a subcurve, define $\sigma(E) = \sigma(\{j \in [n] | p_{j}\in E\})$.
\end{definition}

\begin{proposition}\label{prop:degreeoftail}
\cite[Proposition 3.5]{GJM}
For an allowable choice of $(\gamma,A)$, suppose that $\phi (J , \gamma , A ) \notin \mathbb{Z}$ for any nonempty $J \subset [n]$.  If $X$ is a GIT-semistable curve and $E \subset X$ a tail (a subcurve such that $E \cap \overline{X - E}$ is one point), then $\deg (E) = \sigma (E)$.
\end{proposition}

\begin{corollary}
\cite[Corollary 3.7]{GJM}
Suppose that $\phi (J , \gamma , A ) \notin \mathbb{Z}$ for any $\emptyset \ne J \subset [n]$, and let $E\subseteq X$ be a connected subcurve of $(X,p_1,\ldots,p_n) \in U_{d,n}^{ss}$.  Then
$$ \deg (E) = d - \sum \sigma (Y) $$
where the sum is over all connected components $Y$ of $\overline{X \backslash E}$.
\end{corollary}

Given an $\F$-curve $F(A_{1}, A_{2}, A_{3}, A_{4})$ as above, Proposition \ref{prop:degreeoftail} says that $\deg(L(A_i))=\sigma(A_i)$ if $\phi(A_{i}, \gamma, A)$ is not an integer.  It follows that the degree of the spine is zero, and hence the $\F$-curve is contracted, if and only if $\sum_{i=1}^4 \sigma (A_i) = d$.  If $\phi(A_{i}, \gamma, A) = k$ is an integer, $\deg(L(A_{i}))$ may be either $k$ or $k+1$.  In this case, both curves are identified in the GIT quotient, and it suffices to consider the case where the legs have the maximum possible total degree.


\section{A new modular interpretation for a particular family of  \quotientnames\ }
\label{s:NewSpace}

In this section, we  study the family  $\Udnga{g+1-\ell}{\frac{\ell-1}{\ell+1}}{(\frac{1}{\ell+1})^{2g+2}}$
of birational models for  $\ovop{M}_{0,n}$, where $n= 2(g+1)$ and $1 \le \ell \le g$.    In Theorem \ref{NewSpace} we give a new modular interpretation of them as certain contractions of Hassett
spaces
$$\tau_{\ell}: \ovop{M}_{0,(\frac{1}{\ell+1}-\epsilon)^n} \longrightarrow \Udnga{g+1-\ell}{\frac{\ell-1}{\ell+1}}{(\frac{1}{\ell+1})^{2g+2}},$$
where so called {\em{even chains}}, described in Definition \ref{even}, are replaced by particular curves.   In order to see that the morphisms $\tau_{\ell}$
exist, we first prove Proposition \ref{prop:GITHassett}, which identifies the Veronese quotient associated to a nearby linearization with the Hassett space $\ovop{M}_{0,(\frac{1}{\ell+1}-\epsilon)^n}$.
The results in this section allow us in \S \ref{s:Connection} to prove that nontrivial conformal block divisors $\mathbb{D}(\sL_2,\ell, \omega_1^{2g+2})$ are pullbacks of ample classes from the  $\Udnga{g+1-\ell}{\frac{\ell-1}{\ell+1}}{(\frac{1}{\ell+1})^{2g+2}}$.

\begin{remark}
\label{TrulyNew}
Because their defining linearizations $(\frac{\ell-1}{\ell+1},(\frac{1}{\ell+1})^{2g+2})$ lie on GIT walls, these Veronese quotients $\Udnga{g+1-\ell}{\frac{\ell-1}{\ell+1}}{(\frac{1}{\ell+1})^{2g+2}}$ admit strictly semistable points, and thus their corresponding
moduli functors are not in fact separated.  The quotient described in Theorem \ref{NewSpace} is not isomorphic to a modular compactification in the sense of \cite{Smyth} (cf. Remark \ref{Smyth}).
In fact, the only known method for constructing this compactification is via GIT.  This highlights the strength of the Veronese quotient construction, as we show here that you can use them to construct ``new'' compactifications of $\op{M}_{0,n}$ --
compactifications that have not been described and cannot be described through any of the previously developed techniques.
\end{remark}

\subsection{Defining the maps $\tau_{\ell}$}In this section we define the morphism
$$ \tau_{\ell}: \ovop{M}_{0, (\frac{1}{\ell+1}-\epsilon)^{2g+2}} \to \Udnga{g+1-\ell}{\frac{\ell-1}{\ell+1}}{(\frac{1}{\ell+1})^{2g+2}}.$$
obtained by variation of GIT.
As mentioned above, each of the linearizations $(\frac{\ell-1}{\ell+1},(\frac{1}{\ell+1})^{2g+2})$ lies on a wall.   To show that $\tau_{\ell}$ exists, we will use the general variation of GIT fact that any quotient corresponding to a linearization in a GIT chamber admits a morphism to a quotient corresponding to a linearization on a wall of that chamber.  Namely, in Proposition \ref{prop:GITHassett} below, we identify the \quotientname\ corresponding to a  GIT chamber that is near the GIT wall that contains the linearization $(\frac{\ell-1}{\ell+1},(\frac{1}{\ell+1})^{2g+2})$.  We then use this to describe the morphism to the \quotientname\ $\Udnga{g+1-\ell}{\frac{\ell-1}{\ell+1}}{(\frac{1}{\ell+1})^{2g+2}}$.
We note that the \quotientname\ discussed in Proposition \ref{prop:GITHassett} is in general not a modular compactification of $\op{M}_{0,n}$ but it is isomorphic to one (unlike the quotients described in Theorem \ref{NewSpace}).

\begin{proposition}\label{prop:GITHassett}
For $2 \le \ell \le g-1$, and $\epsilon>0$ sufficiently small, the Hassett space $\ovop{M}_{0, (\frac{1}{\ell+1}-\epsilon)^{2g+2}}$ is isomorphic to the normalization of the \quotientname\
\[
    \Udnga{g+1-\ell}{\frac{\ell-1}{\ell+1} + \epsilon '}{(\frac{1}{\ell+1} - \epsilon)^{2g+2}}.
\]
Here $\epsilon '$ is a positive number that is uniquely determined by the data $d = g+1-\ell$ and $A = (\frac{1}{\ell+1}-\epsilon)^{2g+2}$ (cf. Definition \ref{allowable}).
\end{proposition}

\begin{proof}
By Proposition \ref{prop:MorThroughHassett}, there is a morphism $\phi_{\gamma}:
\ovop{M}_{0, (\frac{1}{\ell+1}-\epsilon)^{2g+2}} \to
\Udnga{g+1-\ell}{\frac{\ell-1}{\ell+1} + \epsilon '}{(\frac{1}{\ell+1} - \epsilon)^{2g+2}}$, which fits into the following commutative diagram:
\[
	\xymatrix{\Mzn \ar[d]_{\rho_{(\frac{1}{\ell+1}-\epsilon)^{2g+2}}} \ar[rd]^{\varphi_{\gamma, (\frac{1}{\ell+1}-\epsilon)^{2g+2}}}\\
	\ovop{M}_{0, (\frac{1}{\ell+1}-\epsilon)^{2g+2}}
	\ar[r]^{\!\!\!\!\!\!\!\!\!\!\!\!\!\!\phi_{\gamma}}&
	\Udnga{g+1-\ell}{\frac{\ell-1}{\ell+1} + \epsilon '}{(\frac{1}{\ell+1} - \epsilon)^{2g+2}}.}
\]
Thus to prove the result, it suffices to show that $\phi_{\gamma}$ is bijective.
Since $(g-\ell)\gamma + (2g+2)(\frac{1}{\ell+1}-\epsilon) = g+2-\ell$,
\[
	\gamma = 1- \frac{2}{\ell+1} + \frac{2(g+1)}{g-\ell}\epsilon.
\]

If $\ell \geq 3$, we have $\gamma > \frac{1}{2}$ and a curve in the semistable locus $U_{g+1-\ell, 2g+2}^{ss}$
does not have multinodal singularities by Lemma \ref{modular}. Similarly, the sum of the weights at a node
cannot exceed $1 - \gamma = \frac{2}{\ell+1} - \frac{2(g+1)}{g-\ell}
\epsilon < 2(\frac{1}{\ell+1} - \epsilon)$.  So at a node, there is at most one marked point.

If $\ell = 2$, $\gamma > \frac{1}{3}$, so a curve in $U_{g+1-\ell, 2g+2}^{ss}$ has at worst a multinodal point of multiplicity $3$.  Note that $1 - (3-1)\gamma = \frac{1}{3} - \frac{4(g+1)}{g-2}\epsilon < \frac{1}{3}-\epsilon$, so by Lemma \ref{modular}, there can be no marked point at a triple point.  Similarly, since $1 - (2-1)\gamma = \frac{2}{3} - \frac{2(g+1)}{g-2}\epsilon < 2(\frac{1}{3}-\epsilon)$, there can be at most one marked point at a node.


To summarize, there is no positive dimensional moduli of curves contracted to the same curve.  In other words, $\phi_{\gamma}$ is an injective map.  The surjectivity comes directly from the properness of both sides.
\end{proof}

\begin{remark}
We make several remarks about Proposition \ref{prop:GITHassett} above.
\begin{enumerate}
	\item In \cite[Theorem 7.1, Corollary 7.2]{GJM}, the authors show that
	for certain values of $\gamma$ and $A$, the corresponding \quotientname\ is $\ovop{M}_{0, A}$.
	Proposition \ref{prop:GITHassett} indicates precise values of $\gamma$
	and $A$ when $A$ is symmetric.
	\item The normalization map of a \quotientname\ is always
	bijective (\cite[Remark 6.2]{GJM}). Thus at least on the level of
	topological
	spaces, the normalization is equal to the \quotientname\
	itself.
	\item If $g \equiv \ell \mbox{ mod } 2$,
	there are strictly semi-stable points on $U_{g+1-\ell, 2g+2}$ for the
	linearization $(\frac{\ell-1}{\ell+1}+\epsilon', (\frac{1}{\ell+1}-
	\epsilon)^{2g+2})$.
	Indeed, for a set $J$ of $g+1$ marked points,
	the weight function
	\[
		\phi(J, \gamma, A) =
		\frac{(g+1)(\frac{1}{\ell+1}-\epsilon)-1}{\frac{2}{\ell+1}
		-\frac{2(g+1)}{g-\ell}\epsilon}
		= \frac{(g-\ell)(g-\ell)-(g-\ell)(g+1)(\ell+1)\epsilon}
		{2(g-\ell)-2(g+1)(\ell+1)\epsilon} = \frac{g-\ell}{2}
	\]
	is an integer.  So the quotient stack $[U_{g+1-\ell,2g+2}^{ss}/\SL(g+2-
	\ell)]$ is not modular in the sense of \cite{Smyth}.
	\item Even if the GIT quotient is modular,
	the moduli theoretic meaning of
	$$\ovop{M}_{0, (\frac{1}{\ell+1}-\epsilon)^{2g+2}} \mbox{ and } \
	\Udnga{g+1-\ell}{\frac{\ell-1}{\ell+1} + \epsilon '}{(\frac{1}{\ell+1} - \epsilon)^{2g+2}}$$ may be
    different in general because on the GIT quotient, multinodal singularities and a
	marked point on a node are allowed.
	But the moduli spaces are nevertheless isomorphic.
	\item The condition on the level is necessary.  Indeed, if $\ell = 1$ or $g$,
	 the GIT quotient is not isomorphic a Hassett space.
\end{enumerate}
\end{remark}

 \subsection{The new modular interpretation}

In this section we will prove Theorem \ref{NewSpace}, which describes the Veronese quotients $\Udnga{g+1-\ell}{\frac{\ell-1}{\ell+1}}{(\frac{1}{\ell+1})^{2g+2}}$ as images of contractions where the so-called {\em{even chains}} in the Hassett
spaces $\ovop{M}_{0,(\frac{1}{\ell+1}-\epsilon)^n}$ are replaced by other curves, described below.

\begin{definition}\label{even}
A curve $(C, x_1 , \ldots , x_{2g+2}) \in \ovop{M}_{0, (\frac{1}{\ell+1}-\epsilon)^{2g+2}}$ is an \emph{\textbf{odd chain}} (resp. \emph{\textbf{even chain}}) if $C$ contains a connected chain $C_1 \cup \cdots \cup C_k$ of rational curves such that:
\begin{enumerate}
	\item each $C_i$ contains exactly two marked points;
	\item each interior component $C_{i}$ for $2 \le i \le k-1$ contains exactly two
	nodes $C_i \cap C_{i-1}$ and $C_i \cap C_{i+1}$;
	\item aside from the two marked points, each of the two end components
	$C_{1}$ and $C_{k}$ contains two ``special'' points, where a special 	
	point is either
	a node or a point at which $\ell+1$ marked points coincide.  
	In the first case, we
	will refer to the connected components of $\overline{C \smallsetminus 
	\cup_{i=1}^{k}C_{i}}$ as ``tails''.
	We will
	regard the second type of special point as a ``tail'' of degree $0$;
	\item the number of marked points on each of tails is odd (resp. even).
\end{enumerate}
\end{definition}

Figure \ref{fig:oddchain} shows two examples of odd chains, when $\ell$ is even.
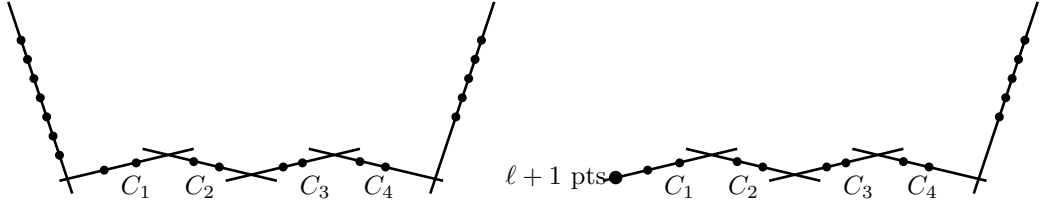
\begin{figure}[!ht]
\begin{center}
\small
\begin{tikzpicture}[scale=0.85]
	\draw[line width=1pt] (0, 10) -- (1, 7);
	\draw[line width=1pt] (0.8, 7.2) -- (2.9, 7.7);
	\draw[line width=1pt] (2.1, 7.7) -- (4.2, 7.2);
	\draw[line width=1pt] (3.4, 7.2) -- (5.5, 7.7);
	\draw[line width=1pt] (4.7, 7.7) -- (6.8, 7.2);
	\draw[line width=1pt] (6.6, 7) -- (7.6, 10);
	\fill (0.2, 9.4) circle (2pt);	
	\fill (0.3, 9.1) circle (2pt);	
	\fill (0.4, 8.8) circle (2pt);	
	\fill (0.5, 8.5) circle (2pt);
	\fill (0.6, 8.2) circle (2pt);	
	\fill (0.7, 7.9) circle (2pt);	
	\fill (0.8, 7.6) circle (2pt);	
	\fill (1.5, 7.366) circle (2pt);
	\fill (2.0, 7.48) circle (2pt);
	\fill (3.3, 7.416) circle (2pt);
	\fill (2.9, 7.5) circle (2pt);
	\fill (4.3, 7.416) circle (2pt);
	\fill (4.6, 7.48) circle (2pt);
	\fill (5.9, 7.416) circle (2pt);
	\fill (5.5, 7.5) circle (2pt);
	\fill (7.4, 9.4) circle (2pt);	
	\fill (7.3, 9.1) circle (2pt);	
	\fill (7.2, 8.8) circle (2pt);	
	\fill (7.1, 8.5) circle (2pt);
	\fill (7.0, 8.2) circle (2pt);
	
	\draw[line width=1pt] (9.3, 7.2) -- (11.4, 7.7);
	\draw[line width=1pt] (10.6, 7.7) -- (12.7, 7.2);
	\draw[line width=1pt] (11.9, 7.2) -- (14.0, 7.7);
	\draw[line width=1pt] (13.2, 7.7) -- (15.3, 7.2);
	\draw[line width=1pt] (15.1, 7) -- (16.1, 10);

	\fill (9.5, 7.25) circle (3pt) node[left] {$\ell +1$ pts};
	\fill (10.0, 7.366) circle (2pt);
	\fill (10.5, 7.48) circle (2pt);
	\fill (11.8, 7.416) circle (2pt);
	\fill (11.4, 7.5) circle (2pt);
	\fill (12.8, 7.416) circle (2pt);
	\fill (13.1, 7.48) circle (2pt);
	\fill (14.4, 7.416) circle (2pt);
	\fill (14.0, 7.5) circle (2pt);
	\fill (15.9, 9.4) circle (2pt);	
	\fill (15.8, 9.1) circle (2pt);	
	\fill (15.7, 8.8) circle (2pt);	
	\fill (15.6, 8.5) circle (2pt);
	\fill (15.5, 8.2) circle (2pt);

	\node (c1) at (2, 7.1) {$C_{1}$};
	\node (c2) at (3, 7.1) {$C_{2}$};
	\node (c3) at (4.8, 7.1) {$C_{3}$};
	\node (c4) at (5.8, 7.1) {$C_{4}$};
	\node (d1) at (10.5,7.1) {$C_{1}$};
	\node (d2) at (11.5, 7.1) {$C_{2}$};
	\node (d3) at (13.3, 7.1) {$C_{3}$};
	\node (d4) at (14.3, 7.1) {$C_{4}$};
\end{tikzpicture}
\normalsize
\end{center}
\caption{Examples of odd chains}\label{fig:oddchain}
\end{figure}

\begin{theorem}
\label{NewSpace}
If $3 \leq \ell \leq g-1$ and $\ell$ is even (resp. odd), then the map $\tau_{\ell}$ restricts to an isomorphism away from the locus of odd chains (resp. even chains).  If $(C, x_1 , \ldots , x_{2g+2} )$ is an odd chain (resp. even chain), then $\tau_{\ell} (C, x_1, \ldots , x_{2g+2} )$ is strictly semistable, and its orbit closure contains a curve where the chain $C_1 \cup \cdots \cup C_k$ has been replaced by a chain $D_1 \cup \cdots \cup D_{k+1}$ with two marked points at each node $D_i \cap D_{i+1}$ (see Figure \ref{fig:contraction}).
\end{theorem}

\begin{figure}[!ht]
\begin{center}
\small
\begin{tikzpicture}[scale=0.85]
	\draw[line width=1pt] (0, 10) -- (1, 7);
	\draw[line width=1pt] (0.8, 7.2) -- (2.9, 7.7);
	\draw[line width=1pt] (2.1, 7.7) -- (4.2, 7.2);
	\draw[line width=1pt] (3.4, 7.2) -- (5.5, 7.7);
	\draw[line width=1pt] (4.7, 7.7) -- (6.8, 7.2);
	\draw[line width=1pt] (6.6, 7) -- (7.6, 10);
	\fill (0.1, 9.7) circle (2pt);
	\fill (0.2, 9.4) circle (2pt);	
	\fill (0.3, 9.1) circle (2pt);	
	\fill (0.4, 8.8) circle (2pt);	
	\fill (0.5, 8.5) circle (2pt);
	\fill (0.6, 8.2) circle (2pt);	
	\fill (0.7, 7.9) circle (2pt);	
	\fill (0.8, 7.6) circle (2pt);	
	\fill (1.5, 7.366) circle (2pt);
	\fill (2.0, 7.48) circle (2pt);
	\fill (3.3, 7.416) circle (2pt);
	\fill (2.9, 7.5) circle (2pt);
	\fill (4.3, 7.416) circle (2pt);
	\fill (4.6, 7.48) circle (2pt);
	\fill (5.9, 7.416) circle (2pt);
	\fill (5.5, 7.5) circle (2pt);
	\fill (7.5, 9.7) circle (2pt);
	\fill (7.4, 9.4) circle (2pt);	
	\fill (7.3, 9.1) circle (2pt);	
	\fill (7.2, 8.8) circle (2pt);	
	\fill (7.1, 8.5) circle (2pt);
	\fill (7.0, 8.2) circle (2pt);
	
	\draw[line width=1pt] (8.5, 10) -- (9.5, 7);
	\draw[line width=1pt] (9.3, 7.2) -- (11.4, 6.7);
	\draw[line width=1pt] (10.6, 6.7) -- (12.7, 7.2);
	\draw[line width=1pt] (11.9, 7.2) -- (14.0, 6.7);
	\draw[line width=1pt] (13.2, 6.7) -- (15.3, 7.2);
	\draw[line width=1pt] (14.5, 7.2) -- (16.6, 6.7);
	\draw[line width=1pt] (16.2, 6.7) -- (17.3, 10);

	\fill (8.6, 9.7) circle (2pt);
	\fill (8.7, 9.4) circle (2pt);	
	\fill (8.8, 9.1) circle (2pt);	
	\fill (8.9, 8.8) circle (2pt);	
	\fill (9.0, 8.5) circle (2pt);
	\fill (9.1, 8.2) circle (2pt);	
	\fill (9.2, 7.9) circle (2pt);	
	\fill (9.3, 7.6) circle (2pt);	
	\fill (17.2, 9.7) circle (2pt);
	\fill (17.1, 9.4) circle (2pt);	
	\fill (17.0, 9.1) circle (2pt);	
	\fill (16.9, 8.8) circle (2pt);	
	\fill (16.8, 8.5) circle (2pt);
	\fill (16.7, 8.2) circle (2pt);

	\fill (11, 6.8) circle (3pt) node[above] {2 pts};	
	\fill (12.3, 7.1) circle (3pt) node[below] {2 pts};	
	\fill (13.6, 6.8) circle (3pt) node[above] {2 pts};	
	\fill (14.9, 7.1) circle (3pt) node[below] {2 pts};	

	\node (t1) at (1, 9) {$T_{1}$};
	\node (t2) at (6.6, 9) {$T_{2}$};
	\node (t3) at (9.5, 9) {$T_{1}$};
	\node (t4) at (16.4, 9) {$T_{2}$};
	\node (c1) at (2, 7.1) {$C_{1}$};
	\node (c2) at (3, 7.1) {$C_{2}$};
	\node (c3) at (4.8, 7.1) {$C_{3}$};
	\node (c4) at (5.8, 7.1) {$C_{4}$};
	\node (a) at (8, 8) {$\Rightarrow$};
\end{tikzpicture}
\normalsize
\end{center}
\caption{The contraction}\label{fig:contraction}
\end{figure}
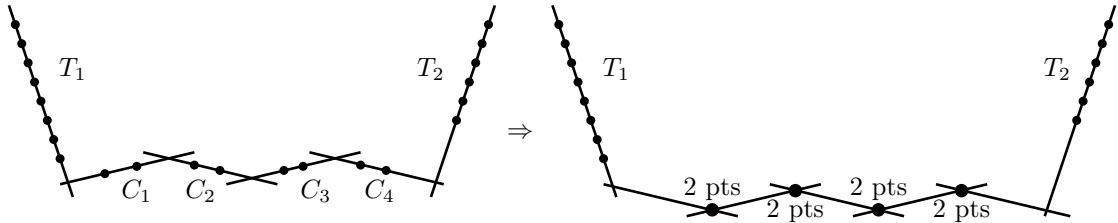

\begin{proof}
Note that both the Hassett space and the GIT quotient are stratified by the topological types of parametrized curves.
Furthermore, $\tau_{\ell}$ is compatible with these stratifications, so $\tau_{\ell}$ contracts a curve $B$ if and only if
\begin{enumerate}
	\item $B$ is in the closure of a stratum,
	\item a general point $(C, x_{1}, \cdots, x_{2g+2})$ of $B$
	has irreducible components $C_1, C_2, \cdots , C_k$
	with positive dimensional moduli,
	\item $C_i$ is contracted by the map from $C$ to $\tau_{\ell}(C, x_{1}, \cdots, x_{2g+2})$ and
	\item the configurations of points on the irreducible components other than the $C_i$'s are fixed.
\end{enumerate}

A component $C_i \subset C \in \ovop{M}_{0,
(\frac{1}{\ell+1}-\epsilon)^{2g+2}}$ has positive dimensional moduli if
it has four or more distinct special points. If a tail with $k$ points is contracted,
then $k (\frac{1}{\ell+1}) \le 1$. But then $k (\frac{1}{\ell+1} - \epsilon) < 1$, so such a tail is impossible on $\ovop{M}_{0, (\frac{1}{\ell+1}-\epsilon)^{2g+2}}$. Thus no tail is contracted.
Now $\gamma = \frac{\ell-1}{\ell+1} \ge \frac{1}{2}$.
By Lemma \ref{modular}, a curve $(D, y_{1}, \cdots, y_{2g+2}) \in
U_{g+1-\ell, 2g+2}^{ss}$ has at worst triplenodes if $\ell = 3$ and
nodes if $\ell \ge 4$. Moreover, the sum of the weights on triplenodes cannot exceed
$1 - 2\gamma = \frac{3-\ell}{\ell+1} \le 0$, so
there are no marked points at a triplenode.
Since $1 - \gamma = \frac{2}{\ell+1}$,
there are at most two marked points at a node. Therefore, the only possible
contracted component is an interior component
$C_i$ with two points of attachment and two marked points.


Now, let $C_i$ be such a component.  Connected to $C_i$ there are two tails $T_{1}$ and $T_{2}$ (not necessarily irreducible), with $i$ and $2g-i$ marked points respectively.  (Here we will regard a point with $\ell+1$ marked points, or equivalently, total weight $1-(\ell+1)\epsilon$ as a `tail' of degree $0$.)
If $i \equiv \ell \mbox{ mod } 2$, then $\phi(T_1) = \frac{i-\ell-1}{2}$ and $\phi(T_2) = \frac{2g-i-\ell-1}{2}$ (see Definition \ref{sigma}), so neither is an integer.  Hence the degree of the component $C_i$ is
\begin{eqnarray*}
	d - (\sigma(T_{1}) + \sigma(T_{2}))
	&=& 	g+1-\ell -(\lceil \frac{i-\ell-1}{2}\rceil
	+ \lceil \frac{2g-i-\ell-1}{2}\rceil) = 1,
\end{eqnarray*}
and thus $C_i$ is not contracted by the map to $\tau_{\ell}(C, x_{1}, \cdots, x_{2g+2})$.
On the other hand, if $i \equiv \ell +1 \mbox{ mod } 2$, then both
$\phi(T_{1}) = \frac{i - \ell -1}{2}$ and $\phi(T_{2}) =
\frac{2g-i-\ell-1}{2}$ are integers.  Therefore, this curve lies in the
strictly semi-stable locus and the image $\tau_{\ell}(C, x_{1}, \cdots, x_{2g+2})$
can be represented by several possible topological types.  By \cite[Proposition 6.7]{GJM}, the orbit closure of $\tau_{\ell}(C, x_{1}, \cdots, x_{2g+2})$ contains a curve in which $C_i$ is replaced by the union of two lines $D_1 \cup D_2$, with two marked points at the node $D_1 \cap D_2$.
\end{proof}

\begin{remark}\label{Smyth}Note that $\tau_{\ell}$ restricts to an isomorphism on the (non-closed) locus of curves consisting of two tails connected by an irreducible bridge with 4 marked points.  But, on the locus of curves consisting of two tails connected by a chain of two bridges with two marked points each, $\tau_{\ell}$ forgets the data of the chain.  The map $\tau_{\ell}$ therefore fails to satisfy axiom (3) of \cite[Definition 1.5]{Smyth}.  In particular, the Veronese quotient described in Theorem \ref{NewSpace}  is not isomorphic to a modular compactification in the sense of \cite{Smyth}.
\end{remark}

\subsection{Morphisms between the moduli spaces we have described}

\begin{corollary}\label{phiell}
If $1 \le \ell \le g-2$, then there is a morphism
\[
	\psi_{\ell, \ell+2}:
	\ovop{M}_{0, (\frac{1}{\ell+1}-\epsilon)^{2g+2}} \to
	\Udnga{g-1-\ell}{\frac{\ell+1}{\ell+3}}{(\frac{1}{\ell+3})^{2g+2}}
\]
preserving the interior.
\end{corollary}

\begin{proof}
For $1 \le \ell \le g-3$, we consider the composition
\[	\ovop{M}_{0, (\frac{1}{\ell+1}-\epsilon)^{2g+2}}
	\to \ovop{M}_{0, (\frac{1}{\ell+3}-\epsilon)^{2g+2}}
	\to \Udnga{g-1-\ell}{\frac{\ell+1}{\ell+3}}{(\frac{1}{\ell+3})^{2g+2}},
\]
where the first morphism is Hassett's reduction morphism \cite[Theorem 4.1]{HassettWeighted} and the last morphism is $\tau_{\ell+2}$.

If $\ell = g-2$, then
$\Udnga{g-1-\ell}{\frac{\ell+1}{\ell+3} + \epsilon '}{(\frac{1}{\ell+3} - \epsilon)^{2g+2}}
= (\PP^{1})^{2g+2}\git \SL(2)$ with symmetric weight datum.  Because there is a morphism
$\ovop{M}_{0, A} \to (\PP^{1})^{2g+2}\git \SL(2)$ for any symmetric weight datum $A$ (\cite[Theorem 8.3]{HassettWeighted}), we obtain $\psi_{g-2, g}$.
\end{proof}

To obtain morphisms between the moduli spaces described in Theorem \ref{NewSpace}, we consider the following diagram.

\begin{displaymath}
	\xymatrix{& \Mzn \ar[ld] \ar[rd] \\
	\ovop{M}_{0, (\frac{1}{\ell+1}-\epsilon)^{2g+2}}
	\ar[rr]\ar[d]^{\cong} \ar[rrdd]^{\psi_{\ell, \ell+2}}
	\ar@/_4pc/[dd]_{\tau_{\ell}}
	&&
	\ovop{M}_{0, (\frac{1}{\ell+3}-\epsilon)^{2g+2}}\ar[d]_{\cong}
	\ar@/^4pc/[dd]^{\tau_{\ell+2}}\\
	\Udnga{g+1-\ell}{\frac{\ell-1}{\ell+1} + \epsilon '}{(\frac{1}{\ell+1} - \epsilon)^{2g+2}} \ar[d]
	&&
	\Udnga{g-1-\ell}{\frac{\ell+1}{\ell+3} + \epsilon '}{(\frac{1}{\ell+3} - \epsilon)^{2g+2}} \ar[d]\\
	\Udnga{g+1-\ell}{\frac{\ell-1}{\ell+1}}{(\frac{1}{\ell+1})^{2g+2}} \ar@{-->}[rr] &&
	\Udnga{g-1-\ell}{\frac{\ell+1}{\ell+3}}{(\frac{1}{\ell+3})^{2g+2}}}
\end{displaymath}

\begin{proposition}\label{prop:morphism}
If $3 \le \ell \le g-2$, then the morphism $\psi_{\ell, \ell+2}$ factors through $\tau_{\ell}$.
\end{proposition}

\begin{proof}
By the rigidity lemma (\cite[Definition-Lemma 1.0]{KeelAnnals}),
it suffices to show that for any curve $B \subset \ovop{M}_{0,
(\frac{1}{\ell+1}-\epsilon)^{2g+2}}$ contracted by $\tau_{\ell}$,
the morphism $\psi_{\ell, \ell+2}$ is constant.  We have already described the curves contracted by $\tau_{\ell}$ in the proof of Theorem \ref{NewSpace}, so it suffices to show that the same curves $B$
are contracted by $\psi_{\ell, \ell+2}$.

When $\ell < g - 2$, $i \equiv \ell +1 \mbox { mod } 2$ if and only if $i \equiv (\ell + 2) + 1 \mbox{ mod } 2$.
So the image of $B$ is contracted by
\[
	\tau_{\ell+2} : \ovop{M}_{0, (\frac{1}{\ell+3}-\epsilon)^{2g+2}}
	\to \Udnga{g-1-\ell}{\frac{\ell+1}{\ell+3}}{(\frac{1}{\ell+3})^{2g+2}}.
\]

If $\ell = g - 2$, then $\psi_{g-2, g}$ is Hassett's reduction morphism
\[
	\ovop{M}_{0, (\frac{1}{g-1}-\epsilon)^{2g+2}} \to
	(\mathbb{P}^{1})^{2g+2}\git SL(2).
\]
In this case there are two types of odd/even chains (of length $1$ or $2$).
It is straightforward to check that these curves
contracted to an isolated singular point of $(\mathbb{P}^{1})^{2g+2}\git SL(2)$ parameterizing strictly semi-stable curves.
\end{proof}

\section{The Veronese quotient divisors $\Vqd{D}_{\gamma,A}$}\label{wpvc}

The main result of this section is Theorem \ref{thm:intersectionFormula}, in which we give a formula for the intersection of $\mathcal{D}_{\gamma, A}$ (explained in Definition \ref{def:DgammaA}) with $\F$-curves on $\ovop{M}_{0,n}$ (described in Definition \ref{FCurve}). As an example of Theorem \ref{thm:intersectionFormula}, in Section \ref{IntForm}, we write down a simple formula for the intersection of the particular divisors
$\mathcal{D}_{\frac{\ell-1}{\ell+1}, (\frac{1}{\ell+1})^{2g+2}}$  with a basis of $\F$-curves.
As a second application of Theorem \ref{thm:intersectionFormula}, in Corollary \ref{class}, we write
down the class of $\mathcal{D}_{\gamma, A}$ in the case that $A$ is $\op{S}_n$-invariant.  Note that we have already described a criterion for determining when these numbers are zero at the end of \S 1.  To compute these numbers in the non-zero case is substantially more complicated.

Let $[n] = A_1 \sqcup A_2 \sqcup A_3 \sqcup A_4$ be a partition and let $F(A_{1}, A_{2}, A_{3}, A_{4})$ be the corresponding $\F$-curve (cf. Definition \ref{FCurve}).   Recall that $\sigma(A_{i})$ is the degree of the leg $L(A_i)$ (cf. Definition \ref{sigma}).  In this section, we establish the following explicit formula for the intersection of  $\Vqd{D}_{\gamma,A}$ and $F(A_{1}, A_{2}, A_{3}, A_{4})$.
\begin{theorem}\label{thm:intersectionFormula}
Given an allowable linearization $(\gamma, A)$ with $d \ge 2$, and an $\F$-curve
$F(A_{1}, A_{2}, A_{3}, A_{4})$:
\begin{eqnarray*}
	F(A_{1}, A_{2}, A_{3}, A_{4}) \cdot \mathcal{D}_{\gamma,A} &=& (\sum_{i=1}^{3}c_{i4}^{2})\frac{w}{2d} +
	(w_{A_{4}} - \frac{w}{d}\sigma(A_{4}))b\\
	&& + \sum_{i=1}^{3}(\frac{w}{d}(\sigma(A_{i})+\sigma(A_{4}))
	-w_{A_{i}}-w_{A_{4}})c_{i4}\\
	&&- \frac{1+\gamma}{2d}(\sum_{i=1}^{4}\sigma(A_{i})
	(d- \sigma(A_{i}))
	- \sum_{i=1}^{3}\sigma(A_{i} \cup A_{4})
	(d-\sigma(A_{i} \cup A_{4})))
\end{eqnarray*}
where $$c_{ij} := d - \sigma ( A_i ) - \sigma ( A_j ) - \sigma ( [n] \backslash
	( A_i \cup A_j ))
	= \sigma ( A_i \cup A_j ) - \sigma ( A_i ) - \sigma ( A_j ),$$
\[
	b = d - \sum_{i=1}^{4}\sigma(A_{i}),
\]
and
$$w = \sum_{i=1}^{n} a_i, \ \mbox{ } \ w_{A_{j}}=\sum_{i \in A_{j}}a_{i}.$$
\end{theorem}

Note that the case of $d = 1$ was studied previously in \cite[Section 2]{AlexeevSwinarski}.  If there is an $A_{i}$ such that $\phi(A_{i}, \gamma, A)$ is an integer, then the $\sigma$ function does not give a unique degree for each leg.  But the result of Theorem \ref{thm:intersectionFormula} is nevertheless independent of the choice of semistable degree distribution.

As an example for how simple this formula can be, consider the following.

\begin{corollary}\label{Jensen}For $n=2g+2$, and $1 \le \ell \le g$,
\begin{displaymath}
F_{n-i-2,i,1,1} \cdot \mathcal{D}_{\frac{\ell-1}{\ell+1}, (\frac{1}{\ell+1})^{2g+2}} = \left\{
\begin{array}{ll}
\frac{1}{\ell+1} & \mbox{ if $i \equiv \ell \pmod{2}$ and $i \geq \ell$,}\\
0 & \mbox{ otherwise.}
\end{array}
\right.
\end{displaymath}
\end{corollary}


Before delving into the proof of Theorem \ref{thm:intersectionFormula} (in Section \ref{proofHBF}), we first explain our approach (in Section \ref{approachHBF}), and develop a tool we will use (in Section \ref{curveliftHBF}), which is a rational lift to $U_{d,n}$ of the image $C$ in $\Udnga{d}{\gamma}{A}$ of a given $\F$-curve.

\subsection{Approach}\label{approachHBF}
Let $C$ be the image of $F(A_{1}, A_{2}, A_{3}, A_{4})$ in $\Udnga{d}{\gamma}{A}$ under the map $\varphi_{\gamma,A}$.  Let $L = \cO(\gamma, A)$ be an allowable polarization on $U_{d,n}$, and let $\overline{L} = L\git_{\gamma,A} \SL(d+1)$ be the associated ample line bundle on $\Udnga{d}{\gamma}{A}$.  By the projection formula, $F(A_{1},A_{2},A_{3},A_{4})\cdot \mathcal{D}_{\gamma,A} = C \cdot \overline{L}$, so to prove Theorem \ref{thm:intersectionFormula} we need to compute $C \cdot \overline{L}$.  To do this, we will lift $C$ to an appropriate curve $\widetilde{C}$ on $U_{d,n}$ and do the intersection there.

\begin{definition}
Let $C$ be a curve in $\Udnga{d}{\gamma}{A}$, and let $\pi : U_{d,n}^{ss} \to \Udnga{d}{\gamma}{A}$ be the quotient map.  A \emph{\textbf{rational lift}} of $C$ to $U_{d,n}$ is a curve $\widetilde{C}$ in $U_{d,n}$ such that
\begin{itemize}
	\item a general point of $\widetilde{C}$ lies in $U_{d,n}^{ss}$;
	\item $\overline{\pi(\widetilde{C})} = C$ and
	$\pi|_{\widetilde{C}} : \widetilde{C} \dashrightarrow C$ is degree 1.
\end{itemize}
\end{definition}

A section of $\overline{L}$ can be pulled-back to a section of $L$ that vanishes on the unstable locus.  It follows that if we have a rational lifting $\widetilde{C}$, then by the projection formula we have
\[
	C \cdot \overline{L} = \widetilde{C} \cdot (L - \sum t_{i}E_{i})
\]
for some rational numbers $t_i > 0$, where the sum is taken over all irreducible unstable divisors.  By the proof of \cite[Proposition 4.6]{GJM}, there are two types of unstable divisors.  One is a divisor of curves with unstable degree distribution and the other is $D_{deg}$, the divisor of curves contained in a hyperplane.  If $\widetilde{C}$ intersects $D_{deg}$ only among unstable divisors, then $C \cdot \overline{L} = \widetilde{C} \cdot (L - t D_{deg})$ for some $t > 0$.

\subsection{An explicit rational lift}\label{curveliftHBF}     In this section we will construct a rational lift $\widetilde{C}$ to $U_{d,n}$ of the image $C$ in $\Udnga{d}{\gamma}{A}$ of an $\F$-curve $F(A_1,A_2,A_3,A_4)$ in $\ovop{M}_{0,n}$.  This lift $\widetilde{C}$ will be used to prove Theorem \ref{thm:intersectionFormula} in Section \ref{proofHBF}.

The total space of an $\F$-curve consists of five components, one of which is the universal curve over $\ov{M}_{0,4}$ and the other four of which are constant families.  We will think of the total space $X \cong \ov{M}_{0,5}$ of spines as the blow-up of 3 points on the diagonal in $\PP^1 \times \PP^1$.  The points of attachment to the legs $L(A_i)$ labeled by $A_1$, $A_2$ and $A_3$ will correspond to the 3 sections of $X$ through the exceptional divisors, while the point of attachment to the leg $L(A_4)$
will correspond to the diagonal.  We denote the classes of the total transforms of two rulings on $\PP^1 \times \PP^1$ by $F$ (for fiber) and $S$ (for section), and the exceptional divisors by $E_i$.  Then on $X \cong \ovop{M}_{0,5}$, the 10 boundary classes are given by
\begin{equation}\label{eqn-curveclassMz5}
\begin{split}
	D_{15} &= S - E_{1}, D_{25} = S - E_{2}, D_{35} = S - E_{3},
	D_{45} = F + S - E_{1} - E_{2} - E_{3},\\
	D_{14} &= F - E_{1}, D_{24} = F - E_{2}, D_{34} = F - E_{3},
	D_{23} = E_{1}, D_{13} = E_{2}, D_{12} = E_{3}.
\end{split}
\end{equation}

We will map $X$ into $\PP^d$ in such a way so that each fiber is the image of the spine of the $\F$-curve.  The general member of this pencil will be GIT-semistable, and the only GIT-unstable divisor meeting this pencil will be $D_{deg}$.  Let $\sigma(A_{i})$ be the degree of the leg containing marked points in $A_{i}$.\footnote{If $U_{d,n}^{ss} = U_{d,n}^{s}$, then the degree of the leg
is uniquely determined by the $\sigma$ function in \cite{GJM},
but if there are strictly semi-stable points, then the degree is not determined
uniquely. In this case we can take any degree distribution which gives
semistable points.}
Then the general fiber must have degree
\[
	b := d - \sum_{i=1}^4 \sigma ( A_i ) .
\]
As the cross-ratio of the 4 points on the spine varies, there are 3 points where the spine breaks into two components.  The degree of one of these components where $A_i$ and $A_j$ come together is exactly
\[
	c_{ij} := d - \sigma ( A_i ) - \sigma ( A_j ) - \sigma ( [n] \backslash
	( A_i \sqcup A_j ))
	= \sigma ( A_i \sqcup A_j ) - \sigma ( A_i ) - \sigma ( A_j ) .
\]
We therefore consider the following divisor class on $X$ (which depends on an integer $a \geq 0$):
\[
	H(a) := aF + bS - \sum_{i=1}^3 c_{i4} E_i .
\]

\begin{lemma}
\label{H(a)bpf}
For $a \gg 0$, $H(a)$ is base-point free.
\end{lemma}

\begin{proof}
Since $X$ is a del Pezzo surface, it is well-known that if $H(a)$ is nef, then $H(a)$ is base-point free.
On $X \cong \ovop{M}_{0,5}$, the cone of curves is generated by the classes $D_{ij}$.  Thus by using \eqref{eqn-curveclassMz5}, it is straightforward to check that $H(a)$ is nef if and only if
\[
	a \ge c_{i4}, b \ge c_{i4}, a+b \ge \sum_{i=1}^{3}c_{i4}.
\]
The second inequality is immediate because $b = c_{12} + c_{34} = c_{13} +
c_{24} = c_{14} + c_{23}$.  So if $a$ is sufficiently large, then $H(a)$ is nef and base-point free.
\end{proof}

\begin{lemma}
\label{Surjectivity}
For $a \gg 0$, the map $H^0 (X,H(a)) \to H^0 (F, H(a) \vert_F )$ is surjective.
\end{lemma}

\begin{proof}
By the exact sequence
\[
	0 \to H^0 (X,H(a)-F) \to H^0 (X,H(a)) \to H^0 (F, H(a) \vert_F ) \to
	H^1 (X,H(a)-F),
\]
it suffices to show that $h^{1}(X,H(a)-F) = 0$.  Since $X$ is a del Pezzo surface, $-K_{X}$ is ample.  Thus
$H(a) - K_{X}$ is ample for $a \gg 0$ by Lemma \ref{H(a)bpf} and $h^{i}(X, H(a)) = h^{i}(X, H(a) - K_{X} + K_{X}) = 0$ for $i > 0$ by the Kodaira vanishing theorem.  Since $H(a) - F = H(a-1)$ by definition, $h^{1}(X, H(a) - F) = 0$ for large $a$ as well.
\end{proof}

By a Riemann-Roch calculation, if $a \gg 0$ then
\[
	h^{0}(X, H(a)) = 3ab - \sum_{i=1}^{3}{c_{i4}+1\choose 2} + 1.
\]
For sufficiently large $a$, $h^{0}(X, H(a))$ is therefore greater than $d + 1$, so we cannot use the complete linear system $|H(a)|$ to construct a map to $\PP^{d}$.  To deal with this problem, we use the following Lemma.

\begin{lemma}
\label{LinearSeries}
Let $V \subset H^0 (X,H(a))$ be a general linear subspace of dimension $h^0 (F, H(a) \vert_F ) +1 = b+2$.  For $a \gg 0$, the map $V \to H^0 (F, H(a) \vert_F )$ is surjective for every fiber $F$.
\end{lemma}

\begin{proof}
For a given fiber $F$, write $K_F$ for the kernel of the map $H^0 (X,H(a)) \to H^0 (F, H(a) \vert_F )$.  By Lemma \ref{Surjectivity}, $K_F$ is a linear space of dimension $h^0 (X,H(a)-F)$.  We will show that dim$V \cap K_F = 1$ for every fiber $F$.  In particular, denote the fiber over a point $y \in \ov{M}_{0,4} \cong \mathbb{P}^{1}$ by $F_y$, and consider the variety
\[
	Z = \{ (y,V) \in \mathbb{P}^1 \times Gr(b+2, H^0 (X,H(a))) \;|\;
	\text{dim}V \cap K_{F_y} \geq 2 \} .
\]
The fibers of $Z$ over $\mathbb{P}^1$ are Schubert varieties, which are known to be irreducible of codimension 2 in the Grassmannian.  It follows that $\mathrm{dim}\; Z < \mathrm{dim}\; Gr(b+2, H^0 (X,H(a))$, and thus $Z$ does not map onto the Grassmannian.  We therefore see that, for the general $V \in Gr(b+2, H^0 (X,H(a)))$, dim$V \cap K_{F_y} < 2$ for every $y \in \mathbb{P}^1$.  On the other hand, we see that dim$V \cap K_{F_y} \geq 1$ trivially for dimension reasons.  It follows that the map $V \to H^0 (F, H(a) \vert_F )$ is surjective for every fiber $F$.
\end{proof}

By Lemma \ref{LinearSeries}, if we consider the map $X \to \PP^{b+1}$ corresponding to the linear series $V$, we see that each individual fiber is mapped to $\PP^{b+1}$ via a complete linear series.  The general fiber therefore maps to a smooth rational normal curve of degree $b$ and the three special fibers map to nodal curves whose two components have the appropriate degrees.  Then, as long as $b<d$, one can embed this $\PP^{b+1}$ in $\PP^d$ and obtain a family of curves in this projective space.

Now consider the case of $b = d$. Because $X$ is a surface, we can take a point $p \in \mathbb{P}^{b+1} \smallsetminus X$ since $d \ge 2$.  Considering a projection from $p$, we obtain a family of curves in $\mathbb{P}^{d}$ with the same degree distribution.  We must choose the point $p$ such that a general member of such a family of curves is semistable. Because it has the correct degree distribution, it suffices to check that a general member of the family is not contained in a hyperplane.  But the image of a curve under projection is degenerate only if the original curve is degenerate.

To each of the $4$ sections we attach a family of curves that does not vary in moduli.  Using the same trick as before, we may take $4$ copies of $\PP^1 \times \PP^1$, mapped into $\PP^d$ via a linear series $V_i \subset \vert \mathcal{O} (x_i ,y_i ) \vert$, where
\[
	x_{i} = \begin{cases}
	H(a) \cdot ( S-E_i ) = a - c_{i4}, & i \ne 4,\\
	H(a) \cdot (F+S-\sum_{j=1}^{3}E_{j}) = a+b-\sum_{j=1}^{3}c_{j4},
	& i = 4
	\end{cases}
\]
and $y_{i} = \sigma ( A_i )$ is the degree of the leg.
Note that if $b = d$, then $\sigma(A_{i}) = 0$ so we don't need to worry about the construction of extra components.

\begin{proposition}
\label{RationalLift}
The family we have constructed is a rational lift of $\varphi_{\gamma,A} (F(A_1 , A_2 , A_3 , A_4 ))$.  It does not intersect any GIT-unstable divisor other than $D_{deg}$.
\end{proposition}

\begin{proof}
We claim that all of the members of this family satisfy the degree conditions required by semi-stability.  Indeed, the general member is a nodal curve with 4 components labeled by the $A_i$'s.  The degree of the leg labeled by $A_i$ is $\mathcal{O} (x_{i},y_{i}) \cdot \mathcal{O} (1,0) = \sigma ( A_i )$ and the degree of the spine is $H(a) \cdot F = b = d - \sum_{i=1}^4 \sigma ( A_i )$.  As one varies the cross-ratio of the 4 points on the spine, there are 3 points where the spine breaks into two components.  The degree of these components are for instance
$H(a) \cdot E_1 = c_{14} = d - \sigma ( A_4 ) - \sigma ( A_1 ) - \sigma ( [n] \backslash ( A_4 \cup A_1 ))$ and
$H(a) \cdot (F - E_{1}) = b - c_{14}= d - \sigma(A_{2})-\sigma(A_{3}) - \sigma([n]\backslash (A_{2} \cup A_{3}))$.
\end{proof}


\subsection{Proof of Theorem \ref{thm:intersectionFormula}}\label{proofHBF}
In this section we prove Theorem \ref{thm:intersectionFormula}, which relies on the curve constructed in Section \ref{curveliftHBF}.

\begin{proof} As is explained in Section \ref{approachHBF}, to prove Theorem \ref{thm:intersectionFormula}, we shall compute the intersection of $C$, the image of $F(A_1,A_2,A_3,A_4)$ in $\Udnga{d}{\gamma}{A}$, with the natural ample line bundle $\overline{L}$.
To do this, it suffices to find a rational lift $\widetilde{C}$ of this curve to $U_{d,n}$ such that a general element of $\widetilde{C}$ is semistable, and compute the intersection upstairs.

By Proposition \ref{RationalLift}, the family constructed in Section \ref{curveliftHBF} is such a lift, so we can carry out these computations with it.  To compute the intersection of
$\widetilde{C}$ with $\mathcal{O}_{Chow} (1)$, fix a general codimension $2$ linear space in $\PP^d$.  The intersection number is precisely the number of curves in the family that intersect this linear space.  In other words, it is the total degree of our $5$ surfaces.  Hence
\begin{eqnarray*}
	\widetilde{C} \cdot \mathcal{O}_{Chow} (1)
	&=& H(a)^2 + \sum_{i=1}^4 \mathcal{O} (x_i , y_i )^2\\
	&=& 2ab - \sum_{i=1}^3 c_{i4}^2 + \sum_{i=1}^{3}
	2(a-c_{i4})\sigma(A_{i})
	+ 2(a+b-\sum_{j=1}^{3}c_{j4})\sigma(A_{4})\\
	&=& 2ad +2\sigma(A_{4})b - \sum_{i=1}^{3}c_{i4}^{2}
	 - \sum_{i=1}^{3}2(\sigma(A_{i})+\sigma(A_{4}))c_{i4}.
\end{eqnarray*}

Similarly, to compute the intersection of $\widetilde{C}$ with $\mathcal{O}_{\mathbb{P}^d_j} (1)$, fix a general hyperplane in $\PP^d$.  The intersection number is precisely the number of points at which the $j$-th section meets this hyperplane.  In other words, it is the degree of the $j$-th section.  If $A_{i}$ is the part of the partition containing $j$, then we see that
\[
	\widetilde{C} \cdot \mathcal{O}_{\PP^d_j} (1)
	= \mathcal{O} (x_i , y_i ) \cdot \mathcal{O} (0,1)
	= x_i =
	\begin{cases}
	a - c_{i4}, & i \ne 4,\\
	a + b - \sum_{k=1}^{3}c_{k4}, & i = 4.
	\end{cases}
\]

One can then easily compute the intersection with $L = \bigotimes_{j=1}^n \mathcal{O}_{\mathbb{P}^d_j} ( a_j ) \otimes \mathcal{O}_{Chow}( \gamma )$ by linearity.
If we denote $\sum_{i \in A_{j}}a_{i}$ by $w_{A_{j}}$ and
$w = \sum_{i=1}^{n} a_i$, then
\begin{eqnarray*}
	\widetilde{C} \cdot L &=&
	\gamma(2ad + 2\sigma(A_{4})b - \sum_{i=1}^{3}c_{i4}^{2}
	- \sum_{i=1}^{3}2(\sigma(A_{i}+A_{4}))c_{i4})\\
	&&+ \sum_{i=1}^{3}w_{A_{i}}(a - c_{i4}) +
	w_{A_{4}}(a+b-\sum_{i=1}^{3}c_{i4})\\
	&=& (2d\gamma + w)a -\sum_{i=1}^{3}c_{i4}^{2}\gamma
	+ (2\sigma(A_{4})\gamma + w_{A_{4}})b\\
	&& - \sum_{i=1}^{3}(2\gamma(\sigma(A_{1})+\sigma(A_{4}))
	+ w_{A_{i}} + w_{A_{4}})c_{i4}.
\end{eqnarray*}

Recall that $C \cdot \overline{L} = \widetilde{C} \cdot (L - tD_{deg})$ for some positive rational number $t$ (Section \ref{approachHBF}).  It remains to determine the value of $t$.  By \cite[Lemma 2.1]{CHS08}, on the moduli space of stable maps $\ovop{M}_{0,0}(\PP^{d},d)$,
\[
	D_{deg} = \frac{1}{2d}\Big(
	(d+1)H - \sum_{k=1}^{\lfloor \frac{d}{2}
	\rfloor}k(d-k)D_{k}\Big),
\]
where $H$ is the locus of stable maps whose image intersects a fixed
codimension 2 linear subspace and $D_{k}$ is the closure of the locus of curves with two irreducible components of degree $k$ and $d-k$, respectively.
If we pull-back $D_{deg}$ by the forgetful map
$f : \ovop{M}_{0,n}(\PP^{d}, d) \to \ovop{M}_{0,0}(\PP^{d},d)$,
then we obtain the same formula for $D_{deg}$ on
$\ovop{M}_{0,n}(\PP^{d}, d)$.
Now for the cycle map $g : \ovop{M}_{0,n}(\PP^{d},d) \to U_{d,n}$,
$g_{*}(D_{deg}) = D_{deg}$, we have $g_{*}(H) = \cO_{Chow}(1)$ and
$g_{*}(D_{k}) = D_{k}$.
Therefore the same formula holds for $U_{d,n}$.

So
\begin{eqnarray*}
	\widetilde{C} \cdot D_{deg} &=&
	\frac{d+1}{2d}\Big(
	2ad +2\sigma(A_{4})b - \sum_{i=1}^{3}c_{i4}^{2}
	 - \sum_{i=1}^{3}2(\sigma(A_{i})+\sigma(A_{4}))c_{i4}
	 \Big)\\
	&&+ \frac{1}{2d}\Big(
	\sum_{i=1}^{4}\sigma(A_{i})(d-\sigma(A_{i}))
	- \sum_{i=1}^{3}(\sigma(A_{i} \cup A_{4}))
	(d - \sigma(A_{i} \cup A_{4})) \Big).
\end{eqnarray*}

Note that the rational lift depends on the choice of $a$.
To obtain an intersection number
$C \cdot \overline{L} = \widetilde{C} \cdot (L - t D_{deg})$ that is independent of the choice of $a$,
the coefficient of $a$ must be $0$. Thus
\[
	2d\gamma + w - t\frac{(d+1)}{2d}2d = 0
\]
and $t = \frac{2d\gamma + w}{1+d} = 1 + \gamma$.

Therefore,
\begin{eqnarray*}
	C \cdot \overline{L} &=& \widetilde{C} \cdot
	(L - (1 + \gamma)D_{deg})\\
	&=& (\sum_{i=1}^{3}c_{i4}^{2})\frac{w}{2d} +
	(w_{A_{4}} - \frac{w}{d}\sigma(A_{4}))b
	+ \sum_{i=1}^{3}(\frac{w}{d}(\sigma(A_{i})+\sigma(A_{4}))
	-w_{A_{i}}-w_{A_{4}})c_{i4}\\
	&&- \frac{1+\gamma}{2d}(\sum_{i=1}^{4}\sigma(A_{i})
	(d- \sigma(A_{i}))
	- \sum_{i=1}^{3}\sigma(A_{i} \cup A_{4})
	(d-\sigma(A_{i} \cup A_{4}))).
\end{eqnarray*}
\end{proof}


\subsection{Example and application of  Theorem \ref{thm:intersectionFormula}}\label{IntForm}
As the $\F$-curves span the vector space of 1-cycles, Theorem \ref{thm:intersectionFormula} gives, in principal, the class of $\Vqd{D}_{\gamma,A}$ in the Ner\'on Severi space.  Using a particular basis (described in Definition \ref{basis}), we explicitly write down the class of $\Vqd{D}_{\gamma,A}$ for $\op{S}_n$-invariant $A$.  The classes depend on the intersection numbers, which as we see below in Example \ref{JensenApp}, are particularly
simple for $\Vqd{D}_{\frac{\ell-1}{\ell+1},(\frac{1}{\ell+1})^{2g+2}}$.

\begin{definition}\label{basis}\cite[Section 2.2.2, Proposition 4.1]{agss} For $1 \le j \le g:= \lfloor \frac{n}{2} - 1 \rfloor$, let $\op{F}_j$ be the $\op{S}_n$-invariant $\op{F}$-curve $F_{1,1,j,n-j-2}$.  The set $\{\op{F}_j : 1 \le j \le g\}$ forms a basis for the group of $1$-cycles $\op{N}_1(\ovop{M}_{0,n})^{\op{S}_{n}}$.
\end{definition}

\begin{definition}\label{DivisorBasis}\cite[Section 3]{KeelMcKernanContractible} For $2 \le j \le \lfloor \frac{n}{2} \rfloor$, let $\op{B}_j$ be the $\op{S}_n$-invariant divisor given by the sum of boundary divisors indexed by sets of size $j$:
$$B_j=\sum_{J \subset [n], |J|=j}\delta_J.$$
The set $\{\op{B}_j : 2 \le j \le g+1\}$ forms a basis for the group of codimension-$1$-cycles $\op{N}^1(\ovop{M}_{0,n})^{\op{S}_{n}}$.
\end{definition}

\begin{corollary}\label{class}  Fix $n=2g+2$ or $n=2g+3$ and $j \in \{1,\cdots, g \}$, and
write $a(\gamma, A)_{ j} = \Vqd{D}_{\gamma,A} \cdot F_{j}$.  If $A$ is an $\op{S}_n$-invariant choice of weights, then $\Vqd{D}_{\gamma,A} \equiv \sum_{r=1}^g b(\gamma, A)_{r} B_{r+1}$, where
$$b(\gamma, A)_r = \sum_{j=1}^{r-1}\Bigg( \frac{r(r+1)}{n-1}-(r-j)\Bigg)a(\gamma, A)_j  + \frac{r(r+1)}{n-1} \sum_{j =r}^g a(\gamma, A)_j, $$
when $n= 2g+3$ is odd, and
$$b(\gamma, A)_r = \sum_{j =1}^{r-1}\Bigg( \frac{r(r+1)}{n-1}-(r-j)\Bigg)a(\gamma, A)_j  + \frac{r(r+1)}{n-1} \sum_{j =r}^{g-1} a(\gamma, A)_j  + \frac{r(r+1)}{2(n-1)}a(\gamma, A)_{g}$$
when $n=2g+2$ is even.
\end{corollary}

\begin{proof}
This follows from the formula given in \cite[Theorem 5.1]{agss}.
\end{proof}

\begin{example}\label{JensenApp}
By the above, we see that
\[
\Vqd{D}_{\frac{\ell-1}{\ell+1},(\frac{1}{\ell+1})^{2g+2}} = \frac{1}{\ell +1} \sum_{r=1}^g \left( \frac{r(r+1)}{n-1} \left(\frac{g-\ell+1}{2}\right) - \lceil \frac{r-\ell+1}{2}\rceil_{+} \lfloor \frac{r-\ell+1}{2}\rfloor_{+}\right)B_{r+1} ,
\]
where
\[
\lceil x \rceil_{+} = \max \{ \lceil x \rceil, 0\}, \quad
\lfloor x \rfloor_{+} = \max \{\lfloor x \rfloor , 0\}.
\]
\end{example}

\begin{proof}
Indeed,
\begin{eqnarray*}
b(\gamma, A)_r &=& \sum_{j =1}^{r-1}\Bigg( \frac{r(r+1)}{n-1}-(r-j)\Bigg)a(\gamma, A)_j  + \frac{r(r+1)}{n-1} \sum_{j =r}^{g-1} a(\gamma, A)_j  + \frac{r(r+1)}{2(n-1)}a(\gamma, A)_{g}\\
&=& \frac{r(r+1)}{n-1}\sum_{j=1}^{g} a(\gamma, A)_{j} -
\frac{r(r+1)}{2(n-1)}a(\gamma, A)_{g} -
\sum_{j=1}^{r-1}(r-j)a(\gamma, A)_{j}.
\end{eqnarray*}
By Corollary \ref{Jensen},
\[
\sum_{j=1}^{g}a(\gamma, A)_{j} =
\begin{cases} \frac{1}{\ell+1}\left(\frac{g-\ell}{2} + 1\right),
& g \equiv \ell \pmod{2},\\
\frac{1}{\ell+1}\left(\frac{g-\ell+1}{2}\right), & g \not\equiv \ell \pmod{2}.
\end{cases}
\]
Also $a(\gamma, A)_{g} = \frac{1}{\ell+1}$ if $g \equiv \ell \pmod{2}$
and zero if $g \not\equiv \ell \pmod{2}$, so we can write
\[
b(\gamma, A)_{r} = \frac{1}{\ell+1}\frac{r(r+1)}{n-1}
\frac{g-\ell+1}{2} - \sum_{j=1}^{r-1}(r-j)a(\gamma, A)_{j}.
\]

By a similar case by case computation, one obtains
\begin{eqnarray*}
\sum_{j=1}^{r-1}(r-j)a(\gamma, A)_{j} &=&
\begin{cases} \left(\frac{r-\ell+1}{2}\right)^{2}, & r \not\equiv \ell \pmod{2} \mbox{ and } \ell \le r-1, \\
\frac{(r-\ell)(r-\ell+2)}{4}, & r \equiv \ell \pmod{2} \mbox{ and } \ell \le r-1,\\
0, &\ell > r-1\end{cases}\\
&=& \lceil \frac{r-\ell+1}{2}\rceil_{+} \lfloor \frac{r-\ell+1}{2}\rfloor_{+}.
\end{eqnarray*}
\end{proof}

\section{Higher level conformal block divisors and \quotientnames\ }
\label{s:Connection}
The main goal of this section is to prove Theorem \ref{main},
which says that when $n = 2g+2$ the divisors $\mathcal{D}_{\gamma, A} = \mathcal{D}_{\frac{\ell-1}{\ell+1}, (\frac{1}{\ell+1})^{2g+2}}$ and $\mathbb{D}(\sL_{2}, \ell, \omega_1^{2g+2})$ determine the same birational models. To prove this, we will show that $\mathbb{D}(\sL_{2}, \ell, \omega_1^{2g+2})$ and $\mathcal{D}_{\gamma, A}$ lie on the same face of the semi-ample cone.

To carry this out, we use a set of $\op{S}_n$-invariant  $\F$-curves, given in Definition \ref{basis}, which were shown in \cite[Proposition 4.1]{agss} to form a basis for $\op{Pic}(\ovop{M}_{0,n})^{\op{S}_n}$.
Using Theorem \ref{thm:intersectionFormula}, we obtained a simple formula for the intersection of these curves with $\mathcal{D}_{\gamma, A}$ in Corollary \ref{Jensen}.  We then show that
$\mathbb{D}(\sL_2,\ell,\omega_{1}^{2g+2})$ is equivalent to a nonnegative combination
of the divisors $\{ \Vqd{D}_{\frac{\ell+2k-1}{\ell+2k+1},(\frac{1}{\ell+2k+1})^{2g+2}} : k \in \Z_{\geq 0}, \ell + 2k \leq g \}$ (Corollary. \ref{poscomb}).
This follows from Proposition \ref{increasing} which shows that the nonzero intersection numbers
$\mathbb{D}(\sL_2,\ell,\omega_{1}^{2g+2}) \cdot F_{i}$ are nondecreasing.

\subsection{Intersection of $F_i$ with $\mathbb{D}(\sL_2,\ell,\omega_1^{2g+2})$ are nondecreasing}\label{increasing}
In this section we prove that the nonzero intersection numbers of $\mathbb{D}(\sL_2,\ell,\omega_1^{2g+2})$ with
the $S_n$ invariant $\F$-curves $F_{i}$ are nondecreasing.

We recall some notation from \cite{ags}.  (Some of this material
appeared in the first version of \cite{ags}, which is still posted on
the arXiv.  However, it was removed from the published version of that paper, and has not been published anywhere else.)

\begin{definition}
For a simple Lie algebra $\mathfrak{g}$, a nonnegative integer $\ell$, and a
sequence $\vec{\lambda} = (\lambda_{1}, \cdots, \lambda_{n})$ of dominant integral weights of $\mathfrak{g}$, let $\mathbb{V}(\mathfrak{g}, \ell, \vec{\lambda})$ be the vector bundle of conformal blocks on $\ovop{M}_{0,n}$. For the precise definition of vector bundle of conformal blocks, see Chapter 3 and 4 of \cite{Ueno}.
\end{definition}

In this paper, we focus on $\mathfrak{g} = \sL_{2}$ cases.

\begin{definition}\label{dfn:rankCBbdl}
We define
\[
r_{\ell}(a_{1}, \cdots, a_{n}) := \mathrm{rank}\; \mathbb{V}(\sL_{2}, \ell, (a_{1}\omega_{1}, \cdots, a_{n}\omega_{1}))
\]
and as a special case,
\[
r_{\ell}(k^{j},t) :=
\operatorname{rank}
\mathbb{V}(\sL_2, \ell, ( {\underset{\text{j times}}{\underbrace{k\omega_{1},\cdots,k\omega_{1}}}},t\omega_{1} )).
\]
\end{definition}

For the basic numerical properties of $r_{\ell}(a_{1}, \cdots, a_{n})$, see \cite[Section 3]{ags}.

\begin{proposition}\label{rankrecurrences}
The ranks $r_{\ell}(1^{j},t)$ are determined by the system of recurrences
\begin{equation}\label{Pascal}  r_{\ell}(1^{j},t)  =   r_{\ell}(1^{j-1},t-1) + r_{\ell}(1^{j-1},t+1),  \qquad t = 1,\cdots,\ell.
\end{equation}
together with seeds
$$r_{\ell}(1^{j},j)=1, \ \mbox{ if  } j \le \ell, \ \ \mbox{and } \ \ r_{\ell}(1^{j},j)=0, \ \mbox{ if  } j > \ell.$$
\end{proposition}

\textit{Remark.} The recurrence (\ref{Pascal}) is somewhat reminiscent of the recurrence for Pascal's triangle.

\begin{proof}
Partition the weight vector $(1,\cdots,1,t)=1^{j}t$ as $1^{j-1} \cup
(1,t)$.  If $j +t$ is odd, then by the odd sum rule
(\cite{ags}*{Proposition 3.5}) $r_{\ell}(1^{j},t) =0$.  So assume $j+t$ is
even.  Then the factorization formula (\cite{ags}*{Proposition 3.3}) states
\begin{equation} \label{factorization}
r_{\ell}(1^{j},t) = \sum_{\mu =0}^{\ell} r_{\ell}(1^{j-1},\mu) r_{\ell}(1, t,  \mu).
\end{equation}
We can simplify this expression.  Recall that by the $\sL_2$ fusion
rules (\cite{ags}*{Proposition 3.4}), $r_{\ell}(1, t, \mu)$ is 0 if $\mu
>t+1$ or if $\mu < t-1$.  Thus the only possibly nonzero summands in
(\ref{factorization}) are when $\mu = t-1$, $t$, or $t+1$.  But when
$\mu = t$, by the odd sum rule, we have $r_{\ell}(1,t,t) =0$.  Thus (\ref{factorization}) simplifies to the following:
\begin{eqnarray*}
r_{\ell}(1^{j},t) & = &  r_{\ell}(1^{j-1},t-1) + r_{\ell}(1^{j-1},t+1) \qquad \qquad t = 1,\cdots,\ell-1;\\
r_{\ell}(1^{j},\ell) & = & r_{\ell}(1^{j-1},\ell-1).
\end{eqnarray*}
Since $r_{\ell}(1^{j-1},\ell+1) =0$, we can unify the two lines above, yielding (\ref{Pascal}).
\end{proof}

\begin{lemma} \label{positive determinants}
Let $i_1 < i_2$ and $j_1 < j_2$.  Suppose $i_1 \equiv i_2 \equiv j_1
\equiv j_2 \pmod{2}$.  Then $r_{\ell}(1^{i_1},j_1) r_{\ell}(1^{i_2},j_2) -
r_{\ell}(1^{i_1},j_2) r_{\ell}(1^{i_2},j_1) \geq 0 $.
\end{lemma}
\begin{proof}
We prove the result by induction on $i_2$.  For the base case, we can
check $(i_1,i_2) = (0,2)$ and $(i_1,i_2) = (1,3)$.  If $(i_1,i_2) =
(0,2)$ the result is true since $r_{\ell}(t) = 0$ if $t > 0$.
Similarly if $(i_1,i_2) =(1,3)$ the result is true since
$r_{\ell}(1,t) =0 $ if $t > 1$.

So suppose the result has been established for all quadruples
$(i_1,i_2,j_1,j_2)$ with $i_2 \leq k-1$.  We apply the recurrence (\ref{Pascal}):
\begin{eqnarray*}
\lefteqn{r_{\ell}(1^{i_1},j_1) r_{\ell}(1^{i_2},j_2) - r_{\ell}(1^{i_1},j_2)
  r_{\ell}(1^{i_2},j_1) } & & \\
&=& \left( \rule{0pt}{14pt} r_{\ell}(1^{i_1-1},j_1-1) + r_{\ell}(1^{i_1-1},j_1+1)
\right) \left( \rule{0pt}{14pt} r_{\ell}(1^{i_2-1},j_2-1) + r_{\ell}(1^{i_2-1},j_2+1)
\right) \\
&& \hspace{0.5in} \mbox{} - \left( \rule{0pt}{14pt} r_{\ell}(1^{i_1-1},j_2-1) + r_{\ell}(1^{i_1-1},j_2+1)
\right) \left( \rule{0pt}{14pt} r_{\ell}(1^{i_2-1},j_1-1) + r_{\ell}(1^{i_2-1},j_1+1)
\right) \\
&=& r_{\ell}(1^{i_1-1},j_1-1) r_{\ell}(1^{i_2-1},j_2-1) - r_{\ell}(1^{i_1-1},j_2-1)
  r_{\ell}(1^{i_2-1},j_1-1) \\
&&  \hspace{0.5in} \mbox{} +  r_{\ell}(1^{i_1-1},j_1-1) r_{\ell}(1^{i_2-1},j_2+1) - r_{\ell}(1^{i_1-1},j_2+1)
  r_{\ell}(1^{i_2-1},j_1-1) \\
&&  \hspace{0.5in} \mbox{} +  r_{\ell}(1^{i_1-1},j_1+1) r_{\ell}(1^{i_2-1},j_2-1) - r_{\ell}(1^{i_1-1},j_2-1)
  r_{\ell}(1^{i_2-1},j_1+1) \\
&&  \hspace{0.5in} \mbox{} +  r_{\ell}(1^{i_1-1},j_1+1) r_{\ell}(1^{i_2-1},j_2+1) - r_{\ell}(1^{i_1-1},j_2+1)
  r_{\ell}(1^{i_2-1},j_1+1).
\end{eqnarray*}
By the induction hypothesis, each of the last four lines is
nonnegative.
\end{proof}

\begin{proposition}\label{increasing}
Suppose $\ell \leq i \leq g-2$ and $i \equiv \ell \pmod{2}$.  Then
$\mathbb{D}(\sL_2,\ell,\omega_{1}^{2g+2}) \cdot F_{i} \leq \mathbb{D}(\sL_2,\ell,\omega_{1}^{2g+2}) \cdot F_{i+2}$.  If, on the other hand, $i \equiv \ell +1 \pmod{2}$, then $\mathbb{D}(\sL_2,\ell,\omega_{1}^{2g+2}) \cdot F_{i} = 0$.
\end{proposition}

\begin{proof}
By \cite[Theorem 4.2]{ags} we have
$\mathbb{D}(\sL_{2},\ell,\omega_{1}^{2g+2}) \cdot F_{i}=r_{\ell}(1^{i},\ell) r_{\ell}(1^{n-i-2},\ell)$.  By the odd sum rule we see that $r_{\ell}(1^{i},\ell)=0$ if $i \equiv \ell +1 \pmod{2}$.

In the remaining cases, we seek to show that
\[ r_{\ell}(1^{i+2},\ell) r_{\ell}(1^{n-i-4},\ell) - r_{\ell}(1^{i},\ell) r_{\ell}(1^{n-i-2},\ell) \geq 0.
\]

We apply the recurrence (\ref{Pascal}) and use $r_{\ell}(1^{j},t)= 0$ if $t
> \ell$ to obtain
\begin{eqnarray*}
\lefteqn{r_{\ell}(1^{i+2},\ell) r_{\ell}(1^{n-i-4},\ell) - r_{\ell}(1^{i},\ell)
r_{\ell}(1^{n-i-2},\ell)} && \\
& = &  \left( \rule{0pt}{14pt} r_{\ell}(1^{i+1},\ell-1)  + r_{\ell}(1^{i+1},\ell+1) \right)
r_{\ell}(1^{n-i-4},\ell) \\
&& \hspace{0.5in} \mbox{} - r_{\ell}(1^{i},\ell) \left( \rule{0pt}{14pt}  r_{\ell}(1^{n-i-3},\ell-1)  +
r_{\ell}(1^{n-i-3},\ell+1) \right) \\
& = &  r_{\ell}(1^{i+1},\ell-1)r_{\ell}(1^{n-i-4},\ell) - r_{\ell}(1^{i},\ell) r_{\ell}(1^{n-i-3},\ell-1) \\
& = &  \left( \rule{0pt}{14pt} r_{\ell}(1^{i},\ell-2) +   r_{\ell}(1^{i},\ell)
\right) r_{\ell}(1^{n-i-4},\ell) -  r_{\ell}(1^{i},\ell) \left( \rule{0pt}{14pt}
  r_{\ell}(1^{n-i-4},\ell-2) +  r_{\ell}(1^{n-i-2},\ell)   \right) \\
& = & r_{\ell}(1^{i},\ell-2) r_{\ell}(1^{n-i-4},\ell) - r_{\ell}(1^{i},\ell)
r_{\ell}(1^{n-i-4},\ell-2).
\end{eqnarray*}

By Lemma \ref{positive determinants}, we have $ r_{\ell}(1^{i},\ell-2)
r_{\ell}(1^{n-i-4},\ell) - r_{\ell}(1^{i},\ell)r_{\ell}(1^{n-i-4},\ell-2) \geq 0$.
\end{proof}

\begin{corollary}\label{poscomb}
The divisor $\mathbb{D}(\sL_2,\ell,\omega_{1}^{2g+2})$ is a nonnegative linear combination
of the divisors $\{ \Vqd{D}_{\frac{\ell+2k-1}{\ell+2k+1},(\frac{1}{\ell+2k+1})^{2g+2}} : k \in \Z_{\geq 0}, \ell + 2k \leq g \}$.  Moreover, the coefficient of $\Vqd{D}_{\frac{\ell-1}{\ell+1},(\frac{1}{\ell+1})^{2g+2}}$ in this expression is strictly positive.
\end{corollary}

\begin{proof}
This follows from Proposition \ref{increasing} and the intersection numbers computed in Corollary \ref{Jensen}.
\end{proof}

\subsection{Morphisms associated to conformal blocks divisors 
}
We are now in a position where we can prove that the divisors $\mathbb{D}(\sL_2,\ell,\omega_1^{2g+2})$ give maps to \quotientnames.

\begin{theorem}\label{main}
The conformal blocks divisor $\mathbb{D}(\sL_2,\ell,\omega_1^{2g+2})$ on $\ovop{M}_{0,2g+2}$ for $1 \le \ell \le g$ is the pullback of an ample class via the morphism
\[
\varphi_{\frac{\ell-1}{\ell+1}, A}=\phi_{\frac{\ell-1}{\ell+1}}\circ \rho_{A}:\ovop{M}_{0,n} \overset{\rho_{A}}{\longrightarrow} \ovop{M}_{0,A} \overset{\phi_{\frac{\ell-1}{\ell+1}}}{\longrightarrow} \Udnga{g+1-\ell}{\frac{\ell-1}{\ell+1}}{A},
\]
where  $A=(\frac{1}{\ell+1})^{2g+2}$, and $\rho_{A}$ is the contraction to Hassett's moduli space $\ovop{M}_{0,A}$
of stable weighted pointed rational curves. \end{theorem}

\begin{proof}
By \cite[Corollary 4.7 and 4.9]{ags} and Corollary \ref{Jensen},
$\Vqd{D}_{\frac{\ell-1}{\ell+1},(\frac{1}{\ell+1})^{2g+2}} \equiv \mathbb{D}(\mathfrak{sl}_{2},\ell, \omega_1^{2g+2})$ if $\ell = 1, 2$.
When $\ell \ge 3$, by Corollary \ref{poscomb},
$\mathbb{D}(\mathfrak{sl}_{2},\ell, \omega_1^{2g+2})$
is a non-negative linear combination of $\Vqd{D}_{\frac{\ell+2k-1}{\ell+2k+1},(\frac{1}{\ell+2k+1})^{2g+2}}$
where $k \in \ZZ_{\ge 0}$ and $\ell + 2k \le g$.
In the latter case, by Proposition \ref{prop:morphism}, we see that all of the divisors in this non-negative linear combination are pullbacks of semi-ample divisors from $\Udnga{g+1-\ell}{\frac{\ell-1}{\ell+1}}{A}$.  Moreover, one of them is ample, and it appears with strictly positive coefficient.  The result follows.
\end{proof}

\begin{remark}\label{rmk:sl2criticallevel}
We note that, for a sequence of dominant integral weights $(k_{1}\omega_{1}, \cdots, k_{n}\omega_{1})$ of $\sL_{2}$, the integer $(\sum_{i=1}^{n}\frac{k_{i}}{2}) - 1$ is called the {\bf critical level} $c\ell$. By \cite[Lemma 4.1]{Fakh}, if $\ell > c\ell$, then $\mathbb{D}(\sL_{2}, \ell, (k_{1}\omega_{1}, \cdots, k_{n}\omega_{1})) \equiv 0$, so the corresponding morphism is a constant map.

When $k_{1} = \cdots = k_{n} = 1$, the critical level is equal to $g$.
Thus it is sufficient to study the cases $1 \le \ell \le g$.  Therefore Theorem \ref{main} is a complete answer for the Lie algebra $\sL_{2}$ and weight data $\omega_{1}^{2g+2}$ cases.
\end{remark}

\section{Conjectural generalizations}\label{conjectures}

Numerical evidence suggests that the connection between \quotientnames\ and type A conformal block divisors holds in a more general setting.  In this section, we provide some of this evidence and make a few conjectures.

 \subsection{$\sL_{2}$ cases}\label{5a}

In this section, we consider $\sL_{2}$ symmetric weight cases, i.e,
$\mathbb{D}(\sL_{2}, \ell, k\omega_{1}^{n})$ for
$1 \le k \le \ell$. Theorem \ref{main} tells us that when $k = 1$, the associated birational models are \quotientnames.  Before we can predict the birational models associated to other conformal block divisors, we need the following useful lemma.

\begin{lemma}\cite[Lemma 3.16]{ags}\label{lem:nonzeroCBspace}
The rank $r_{\ell}(a_{1}, \cdots, a_{n}) \neq 0$ if and only if $\Lambda = \sum_{i=1}^{n} a_i$ is even and, for any subset $I \subset \{ 1, \ldots n \}$ with $n - \vert I \vert$ odd, we have
\[
\Lambda - (n-1) \ell \leq \sum_{i \in I} (2 a_i - \ell ) .
\]
\end{lemma}

Note that for a given weight datum, the left-hand side of this expression is fixed, while the right-hand side is minimized by summing over all weights such that $2 a_i < \ell$.

The next result shows that when $k = \ell$, we get the same birational model as in the case of $k = 1$.

\begin{proposition}
We have the following equalities between conformal block divisors:
\[
\mathbb{D} (\sL_2 , \ell , \ell\omega_{1}^{n}) = \ell \mathbb{D} (\sL_2 , 1, \omega_{1}^{n})=\frac{\ell}{k}\mathbb{D} (\sL_{2 k} , 1, \omega_{k}^{n}).
\]
\end{proposition}

\begin{proof}
The second assertion is a direct application of \cite[Proposition 5.1]{GiansiracusaGibney}, which says that
$$\mathbb{D} (\sL_{r} , 1, (\omega_{z_1}, \cdots , \omega_{z_n}))=\frac{1}{k}\mathbb{D} (\sL_{rk} , 1, (\omega_{k z_1}, \cdots , \omega_{k z_n})).$$

For the first assertion, let $\mathbb{D} = \mathbb{D} (\sL_2 , \ell , \ell \omega_{1}^{n})$.  It suffices to consider intersection numbers of $\mathbb{D}$ with F-curves of the form $F_{i} = F_{n-i-2,i,1,1}$.
Then
$$ \mathbb{D} \cdot F_{i_1, i_2 , i_3 , i_4} = \sum_{\vec{u} \in P_{\ell}^4} \mathrm{deg}( \mathbb{V}(\sL_{2}, \ell, (u_{1}\omega_{1}, u_{2}\omega_{1}, u_{3}\omega_{1}, u_{4}\omega_{1})) \prod_{k=1}^4 r_{\ell}(\ell^{i_k} ,t),$$
where $P_{\ell} = \{0, 1, \cdots, \ell\}$.
In the case where $i_3 = i_4 = 1$, we may use the two-point fusion rule for $\sL_2$ to obtain:
$$ \mathbb{D} \cdot F_{i} = \sum_{0 \le u_1 , u_2\le \ell} \mathrm{deg} ( \mathbb{V}(\sL_{2}, \ell, (u_1\omega_{1} , u_2\omega_{1} , \ell\omega_{1} , \ell\omega_{1}) ) r_{\ell}( \ell^{n-i-2}, u_1 ) r_{\ell}( \ell^i , u_2 ).$$
By the case $I = \{ n \}$ if $n$ is even and $I = \emptyset$ if $n$ is odd in Lemma \ref{lem:nonzeroCBspace}, we see that $r_{\ell}(\ell^j,t) = 0$ if $0 < t < \ell$.  Hence
\[
\mathbb{D} \cdot F_{i} = \sum_{u_1 , u_2 = 0, \ell} \mathrm{deg} ( \mathbb{V}(\sL_{2}, \ell, (u_1\omega_{1} , u_2\omega_{1} , \ell\omega_{1} , \ell\omega_{1}) ) r_{\ell}( \ell^{n-i-2}, u_1 ) r_{\ell}( \ell^i , u_2 ).
\]
But by \cite[Proposition 4.2]{Fakh}, $\deg (\mathbb{V}(\sL_{2}, \ell,
(0, 0, \ell \omega_{1}, \ell \omega_{1}))) = \deg (\mathbb{V}(\sL_{2}, \ell,
(0, \ell \omega_{1}, \ell \omega_{1}, \ell \omega_{1}))) = 0$ and
$\deg (\mathbb{V}(\sL_{2}, \ell,
(\ell \omega_{1}, \ell \omega_{1}, \ell \omega_{1}, \ell \omega_{1}))) =\ell$.
Thus
\[
\mathbb{D} \cdot F_{i} = \ell r_{\ell}( \ell^{n-i-2}, \ell ) r_{\ell}( \ell^i , \ell ).
\]
It therefore suffices to show that $r_{\ell}(\ell^t,\ell) = r_{\ell}(1^t,1)$, but this follows by induction via the factorization rules and the propagation of vacua.
\end{proof}

For the majority of values of $k$ such that $1 < k < \ell$, the divisor $\mathbb{D} (\sL_2 , \ell , k\omega_{1}^{n})$ appears to give a map to a Hassett space.  To establish our evidence for this, we first start with a lemma.

\begin{lemma}\label{lem:zeroCBspace}
Suppose that $1 < k < \ell$.  Then $r_{\ell}(k^i , t )=0$ if and only if either $ki+t$ is odd or one of the following holds:
\begin{enumerate}
\item  $2k \le \ell$ and $i < \max \{ \frac{t}{k} , 2- \frac{t}{k} \}$;
\item  $2k > \ell$, $i$ is even, and $i < \max \{ \frac{t}{\ell-k} , 2- \frac{t}{\ell-k} \}$;
\item  $2k > \ell$, $i$ is odd, and $i < \max \{ \frac{\ell-t}{\ell-k} , 2- \frac{\ell-t}{\ell-k} \}$.
\end{enumerate}
\end{lemma}

\begin{proof}
Each of these follows from case by case analysis of Lemma \ref{lem:nonzeroCBspace} above and the remark that follows it.
\end{proof}

We now consider which symmetric F-curves have trivial intersection with the divisors in question.

\begin{proposition}\label{prop:Fcurvecontraction}
Suppose that $1 < k < \frac{3}{4} \ell$ and let $\mathbb{D} = \mathbb{D} (\sL_2 , \ell , k\omega_{1}^{n})$.  Assume that $n$ is even and $\ell \le \frac{kn}{2} - 1$.  (Recall that this is necessary for the non-triviality of $\mathbb{D}$ by remark \ref{rmk:sl2criticallevel}.)  If $a \leq b \leq c \leq d$, then $\mathbb{D} \cdot F_{a,b,c,,d} = 0$ if and only if $a+b+c \leq \frac{\ell +1}{k}$.
\end{proposition}

\begin{proof}
By \cite[Proposition 4.7]{Fakh}, the map associated to $\mathbb{D}$ factors through the map $\ovop{M}_{0,n} \to \ovop{M}_{0,(\frac{k}{\ell+1})^{n}}$. It follows that, if $a+b+c \leq \frac{\ell +1}{k}$, then $\mathbb{D} \cdot F_{a,b,c,d} = 0$.  It therefore suffices to show the converse.  We assume throughout that $a+b+c > \frac{\ell +1}{k}$.

By \cite[Proposition 2.7]{Fakh}, we have
$$ \mathbb{D} \cdot F_{a,b,c,d} = \sum_{\vec{u} \in P_{\ell}^{4}} \deg ( \mathbb{V}(\sL_{2}, \ell, (u_{1}\omega_{1}, u_{2}\omega_{1}, u_{3}\omega_{1}, u_{4}\omega_{1})) r_{\ell}(k^a , u_1 ) r_{\ell}(k^b , u_2 )r_{\ell}(k^c , u_3 ) r_{\ell}(k^d , u_4 ) .$$
Since each term in the sum above is nonnegative, it suffices to show that a single term is nonzero.

We first consider the case that $2k \leq \ell$.  We set
\begin{displaymath}
w_a = \left\{ \begin{array}{ll}
\min \{ ka, \ell \} & \textrm{if $ka \equiv \ell \pmod{2}$},\\
\min \{ ka, \ell-1 \} & \textrm{if $ka \not\equiv \ell \pmod{2}$}.
\end{array} \right.
\end{displaymath}
Note that by assumption, both $k(a+b+c+d)= kn$ and $k(a+b+c) + \ell$ are strictly greater than $2 \ell + 1$. So it is straightforward to check that $w_{a}+w_{b}+w_{c}+w_{d} > 2 \ell$ and $\ell +1 > w_{d}$.  Note further that $2 \ell + 2+ 2w_{a} > 2w_a + w_c + w_d$ and $2w_{a} \ge 4$. It follows that there is an integer $w_b'$ such that $w_b' \equiv w_b \pmod{2}$, $2 \ell < w_a + w_b' + w_c + w_d < 2 \ell + 2 + 2w_a$. Then $w_{a}+w_{b}'+w_{c}+w_{d} \equiv w_{a}+w_{b}+w_{c}+w_{d} \equiv k(a+b+c+d) \equiv 0 \pmod{2}$. Thus $\deg ( \mathbb{V}(\sL_{2}, \ell, (w_a \omega_{1},w_b' \omega_{1},w_c\omega_{1},w_d\omega_{1}) )) \neq 0$ by \cite[Proposition 4.2]{Fakh}.  It therefore suffices to show that $r_{\ell}(k^a , w_a) \neq 0$.  But in this case, Lemma \ref{lem:zeroCBspace} tells us that $r_{\ell}(k^a , w_a)=0$ only if $a < \max \{\frac{w_{a}}{k}, 2-\frac{w_{a}}{k}\} \le \max \{a, 2\}$, which is possible only if $a = 1$. But then $w_{a} = k$ so $r_{\ell}(k^{a}, w_{a}) = r_{\ell}(k, k) = 1 \ne 0$ by the two point fusion rule. Therefore it is always nonzero.

We next consider the case that $2k > \ell$.  Note that in this case, $\frac{\ell+1}{k} < 3$, so we must show that no F-curves are contracted.  We set
\begin{displaymath}
w_a = \left\{ \begin{array}{ll}
k & \textrm{if $a$ is odd},\\
2( \ell -k) & \textrm{if $a$ is even}.
\end{array} \right.
\end{displaymath}
Again we have that $\ell + 1 > \max \{ w_a , w_b , w_c , w_d \}$, $2 \ell < w_a + w_b + w_c + w_d < 2 \ell + 2 + 2 \min \{ w_a , w_b , w_c , w_d \}$, and $w_{a}+w_{b}+w_{c}+w_{d} \equiv 0 \pmod{2}$. Thus $\deg ( \mathbb{V}(\sL_{2}, \ell, (w_{a}\omega_{1}, w_{b}\omega_{1}, w_{c}\omega_{1}, w_{d}\omega_{1}))) \neq 0$ by \cite[Proposition 4.2]{Fakh}.  If $a$ is odd, we see that $r_{\ell} (k^a , w_a ) = 0$ if and only if $a < 1$.  If $a$ is even, we see that $r_{\ell} (k^a , w_a ) = 0$ only if $a< 2$.  It follows that no F-curves are contracted.
\end{proof}

By Proposition \ref{prop:Fcurvecontraction}, we see that the $\F$-curves that have trivial intersection with $\mathbb{D}(\sL_{2}, \ell, k\omega_{1}^{n})$ are precisely those that are contracted by the morphism $\rho_{(\frac{k}{\ell+1})^n} : \ovop{M}_{0,n} \to \ovop{M}_{0, (\frac{k}{\ell+1})^{n}}$.  At the present time, this is not sufficient to conclude that $\mathbb{D}(\sL_{2}, \ell, k\omega_{1}^{n})$ is in fact the pullback of an ample divisor from this Hassett space, although this would follow from a well-known conjecture (see \cite[Question 1.1]{KeelMcKernanContractible}).


\begin{theorem}\label{cnj:sl2conjecture}
Assume that $n$ is even, $1 < k < \frac{3}{4} \ell$, and $\ell \le \frac{kn}{2} - 1$.  If the $\F$-Conjecture holds (see \cite[Question 1.1]{KeelMcKernanContractible}), then the divisor $\mathbb{D}(\sL_{2}, \ell, k\omega_{1}^{n})$ is the pullback of an ample class via the morphism $\rho_{(\frac{k}{\ell+1})^n} : \ovop{M}_{0,n} \to \ovop{M}_{0, (\frac{k}{\ell+1})^{n}}$.
In particular, if $\frac{\ell+1}{3} < k < \frac{3}{4}\ell$, then $\mathbb{D}$ is ample.
\end{theorem}

We note further that Proposition \ref{prop:Fcurvecontraction} does not cover all of the possible cases of symmetric-weight $\sL_{2}$ conformal block divisors.  In particular, if $k \geq \frac{3}{4} \ell$, then $\mathbb{D}(\sL_{2}, \ell, k\omega_{1}^{n})$ may in fact have trivial intersection with an $\F$-curve if all of the legs contain an even number of marked points.  Such is the case, for example, of the divisor $\mathbb{D}(\sL_{2}, 4, 3\omega_{1}^{8})$.  This divisor has zero intersection with the $\F$-curve $F(2,2,2,2)$ and positive intersection with every other $\F$-curve.  It is not difficult to see that the associated birational model is the Kontsevich-Boggi compactification of $\op{M}_{0,8}$ (see \cite[Section 7.2]{GJM} for details on this moduli space).

\subsection{Birational properties of type A conformal blocks}\label{5b}

We note that in every known case the birational model associated to conformal blocks divisors is in fact a compactification of $\op{M}_{0,n}$.  In other words, the associated morphism restricts to an isomorphism on the interior.  We posit this as a conjecture.


\begin{conjecture}
\label{PreserveInterior}
Let $\mathbb{D}$ be a non-trivial conformal blocks divisor of type A with strictly positive weights. Then $\mathbb{D}$ separates all points on $\op{M}_{0,n}$.
More precisely, for any two distinct points $x_{1}, x_{2} \in \op{M}_{0,n}$, the morphism
	\[
		H^{0}(\ovop{M}_{0,n}, \mathbb{D}) \to
		\mathbb{D}|_{x_{1}}\oplus \mathbb{D}|_{x_{2}}
	\]
is surjective.
\end{conjecture}

If true, this conjecture would have several interesting consequences.  Among them is the following simple description of the maps associated to conformal blocks divisors.  Let $\mathbb{D}$ be a conformal blocks divisor of type A and $\rho_{\mathbb{D}} : \ovop{M}_{0,n} \to X$ the associated morphism.  Consider a boundary stratum
$$ \prod_{i=1}^m \op{M}_{0,k_i} \hookrightarrow \ovop{M}_{0,n} .$$

By factorization \cite[Proposition 2.4]{Fakh}, the pullback of a type A conformal blocks divisor to $\ovop{M}_{0,k_m}$ and its interior $M_{0,k_{m}}$ is an effective sum of type A conformal blocks divisors.  If all of the divisors in this sum are trivial, then the restriction of $\rho_{\mathbb{D}}$ to this boundary stratum forgets a component of the curve:
\[
\xymatrix{
\prod_{i=1}^m \op{M}_{0,k_i} \ar[r] \ar[d] & \ovop{M}_{0,n} \ar[d]^{\rho_{\mathbb{D}}} \\
\prod_{i=1}^{m-1} \op{M}_{0,k_i} \ar[r] & X.}
\]

If the only non-trivial divisors in this sum have weight zero on some subset of the attaching points, then these divisors are pullbacks of non-trivial conformal blocks divisors via the map that forgets these points.  Hence, the restriction of $\rho_{\mathbb{D}}$ to this boundary stratum forgets these attaching points:
\[
	\xymatrix{
\op{M}_{0,k_i} \ar[r] \ar[d] & \ovop{M}_{0,n} \ar[d]^{\rho_{\mathbb{D}}} \\
\op{M}_{0,k_i -j} \ar[r] & X.}
\]

Finally, if any of the non-trivial conformal blocks divisors in this sum has strictly positive weights, then by Conjecture \ref{PreserveInterior}, the restriction of $\rho_{\mathbb{D}}$ to the interior of this stratum is an isomorphism:
\[
	\xymatrix{
    \op{M}_{0,k_i} \ar[r] \ar[d]^{\cong} & \ovop{M}_{0,n} \ar[d]^{\rho_{\mathbb{D}}} \\
	\op{M}_{0,k_i} \ar[r] & X.}
\]

In summary, Conjecture \ref{PreserveInterior} implies that the image of a boundary stratum $\prod_{i=1}^{m}\op{M}_{0, k_{i}}$ in $X$ is isomorphic to
\[
	\prod_{i=1}^{a}\op{M}_{0, k_{i}} \times \prod_{i=a+1}^{b}\op{M}_{0, k_{i}-j_{i}}
\]
for some $1 \le a \le b \le n$ and $1 \le j_{i} \le k_{i}-3$.

In this way, the morphisms associated to type A conformal blocks are somewhat reminiscent of Smyth's modular compactifications (see \cite{Smyth}).  Each of Smyth's compactifications can be described by assigning, to each boundary stratum, a collection of ``forgotten'' components.  In a similar way, the morphism $\rho_{\mathbb{D}}$ appears to assign to each boundary stratum a collection of forgotten components and forgotten points of attachment.  It follows that, if Conjecture \ref{PreserveInterior} holds, one can understand the morphism $\rho_{\mathbb{D}}$  completely from such combinatorial data.


%
%

\end{document}